\documentclass[11pt]{article}
 
\usepackage[margin=0.8in]{geometry} 
\usepackage{amsmath,amsthm,amssymb}
\usepackage{mathrsfs}
\usepackage[T1]{fontenc}
\usepackage[english]{babel}
\usepackage{graphicx}
\usepackage{color} 

\usepackage{hyperref}
\usepackage{afterpage}

\usepackage{color}

\usepackage{atbegshi}% http://ctan.org/pkg/atbegshi
\AtBeginDocument{\AtBeginShipoutNext{\AtBeginShipoutDiscard}}
\usepackage{pdfpages} % Para por pdf no inicio
\usepackage{longtable} % longtable's, duh
\usepackage{amsmath,amsthm,amsfonts,amssymb,relsize,bm} %maths stuff
\usepackage{xparse,xcolor} % cores
\usepackage{graphicx,float} %imagens, duh
\usepackage{etoolbox}
\usepackage[shortlabels]{enumitem} %enumerate environment
\usepackage{tikz-cd} %Diagramas Comutativos
\usepackage{cite} %Para aglomerar citacoes
\usepackage[normalem]{ulem}
\usepackage{comment}
\usepackage{tocloft}
\numberwithin{equation}{section}

\theoremstyle{plain}
\newtheorem{theorem}{Theorem}[section]

\newtheorem*{theorem*}{Theorem}
\newtheorem{lemma}{Lemma}[section]
\newtheorem{corollary}{Corollary}[section]
\newtheorem{proposition}{Proposition}[section]
\newtheorem{claim}{Claim}[section]

\theoremstyle{definition}

\newtheorem*{definition*}{Definition}

\begin{document}
  \title{On the uniform distribution of the zero ordinates of the $L-$function
associated with $\theta(z)^{-3}\,\eta(2z)^{12}$}

\author{{Pedro Ribeiro}} %Department of Mathematics, Faculty of  Sciences,  University of Porto, \\ Rua do Campo Alegre,  687; 4169-007 Porto (Portugal)}}  
\thanks{


Department of Mathematics, Faculty of Sciences of University of Porto, Rua do Campo Alegre,  687; 4169-007 Porto (Portugal). 

\,\,\,\,\,E-mail: pedromanelribeiro1812@gmail.com}
\date{}
 
\maketitle

\begin{abstract}
We show that the ordinates of the nontrivial zeros of certain $L-$functions
attached to half-integral weight cusp forms are uniformly distributed
modulo one. 

%Moreover, we find a bound for the distance between the %consecutive ordinates of the critical zeros of the same %class of $L-$functions.  
\end{abstract}

\tableofcontents 

%\thispagestyle{empty}
% \clearpage
\pagenumbering{arabic}

\section{Introduction}

A sequence of real numbers $(\gamma_{n})_{n\in\mathbb{N}}$ is said
to be uniformly distributed modulo 1 if, for any pair of real numbers
$0\leq a<b\leq1$, we have the relation

\begin{equation}
\lim_{N\rightarrow\infty}\frac{\#\left\{ n\leq N:\,a\leq\{\gamma_{n}\}<b\right\} }{N}=b-a, \label{modulo one}
\end{equation}
where $\{x\}$ denotes the fractional part of $x$. A necessary and
sufficient condition for a sequence to be uniformly distributed modulo
1 is Weyl's criterion: it says that a sequence $(\gamma_{n})_{n\in\mathbb{N}}$ satisfies (\ref{modulo one}) if and only if
\begin{equation}
\sum_{n=1}^{N}e^{2\pi im\gamma_{n}}=\text{o}(N),\label{to prove Weyl criterion}
\end{equation}
for every $m\in\mathbb{Z}\setminus\{0\}$.

\bigskip{}

The study of the uniform distribution of the ordinates (i.e., the
imaginary parts) of the nontrivial zeros of $\zeta(s)$, $(\gamma_{n})_{n\in\mathbb{N}}$,
was inaugurated by Rademacher in 1956 \cite{rademacher}. Rademacher's
result, which established the uniform distribution of $\left(\alpha\gamma_{n}\right)_{n\in\mathbb{N}}$,
$\alpha\neq0$, was dependent on the truth of the Riemann hypothesis.
Elliott and, later, Hlawka obtained the analogous unconditional result.
Hlawka also considered upper bounds for the discrepancy of the sequence $\left(\alpha\gamma_{n}\right)_{n\in\mathbb{N}},\,0<\gamma_{n}\leq T$.
For more interesting historical developments on this subject, we refer
to the introduction of \cite{akbary_murty} and to the fourth section
of \cite{steuding_100}.

Akbary and Murty extended much of these results to a class of Dirichlet
series containing the Selberg class. Their results were also extended
in new directions by several authors. For example, Garunk\v{s}tis, Steuding
and \v{S}im\.{e}nas proved the uniform distribution of the $a-$points of
the Selberg zeta function \cite{selberg_zeta_apoints}. When $a\neq1$, the study of the uniform
distribution of the $a-$points of some L functions in the Selberg
Class was also considered by Jakhlouti, Mazhouda and Steuding \cite{steuding_a_points}. Concerning the Riemann $\zeta-$function, Baluyot and Gonek found discrepancy estimates for the ordinates of its $a-$points, $(\lambda \gamma_{a,n})_{n\in \mathbb{N}}$. 

Although the existence of an Euler product is an important feature
of the functions in the Selberg class as well as in the class of Dirichlet
series considered in \cite{akbary_murty}, it is possible to study the uniform distribution
of the zero ordinates of some Dirichlet series that do not possess
this property. Regarding the zeros of Epstein zeta functions, Schmeller
\cite{schmeller} was able to prove a conditional result concerning
their uniform distribution. One may notice, however, that for binary
quadratic forms, Schmeller's result is no longer conditional, thanks
to a mean value theorem due to M\"uller \cite{muller}.

\bigskip{}

In this paper, we study the uniform distribution of the zero ordinates
of $L-$functions attached to some
cusp forms having half-integral weight. To get started, let $k\in\mathbb{N}$ and denote by $S_{k+\frac{1}{2}}\left(\Gamma_{0}(4)\right)$
the space of cusp forms of weight $k+\frac{1}{2}$ on $\Gamma_{0}(4)$. This means that any element $f(z)\in S_{k+\frac{1}{2}}\left(\Gamma_{0}(4)\right)$
satisfies\footnote{Throughout this paper we define $\sqrt{z}=z^{1/2}$ so that $-\frac{\pi}{2}<\arg(z^{1/2})\leq\frac{\pi}{2}$.}
\begin{equation}
f\left(\frac{az+b}{cz+d}\right)=\left(\frac{c}{d}\right)^{2k+1}\epsilon_{d}^{-2k-1}\left(cz+d\right)^{k+\frac{1}{2}}f(z),\label{modular properties half integral weight}
\end{equation}
whenever $\left(\begin{array}{cc}
a & b\\
c & d
\end{array}\right)\in\Gamma_{0}(4)$. Here, $\left(\frac{c}{d}\right)$ is Shimura's extension of the
Jacobi symbol (see \cite{shimura_half}, p. 442{]}) and $\epsilon_{d}$ is
defined by
\begin{equation}
\epsilon_{d}=\begin{cases}
1 & d\equiv1\,\,\mod4\\
i & d\equiv3\,\,\mod4
\end{cases}.\label{character shimura epsilon d}
\end{equation}
The theory of modular forms of half-integral weight was extensively
developed by Shimura and we refer to the seminal paper \cite{shimura_half}
for some classical facts about these modular forms. Since $\left(\begin{array}{cc}
1 & 1\\
0 & 1
\end{array}\right)\in\Gamma_{0}(4)$, $f(z)$ has a Fourier expansion of the form
\begin{equation}
f(z)=\sum_{n=1}^{\infty}a_{f}(n)\,e^{2\pi inz},\label{Fourier expansion of cusp intro}
\end{equation}
where $z\in\mathbb{H}=\left\{ w\in\mathbb{C}\,:\,\text{Im}(w)>0\right\} $.
We can attach a Dirichlet series $L(s,f)$ to $f(z)$,
\begin{equation}
L(s,f)=\sum_{n=1}^{\infty}\frac{a_{f}(n)}{n^{s}},\label{Dirichlet L series cusp form at intro first}
\end{equation}
and show that, for sufficiently large $\text{Re}(s)$, this series
converges absolutely. The usual argument invoked to study the analytic
continuation of $L-$functions of cusp forms with integral weight
can also be applied to the half-integral case. Thus, $L(s,f)$ can
be analytically continued to the whole complex plane $\mathbb{C}$
as an entire function of $s$. Moreover, it will satisfy Hecke's functional equation, 
\begin{equation}
\pi^{-s}\Gamma\left(s\right)\,L\left(s,f\right)=\pi^{-\left(k+\frac{1}{2}-s\right)}\Gamma\left(k+\frac{1}{2}-s\right)\,L\left(k+\frac{1}{2}-s,f|W_{4}\right),\label{functional equation first Cusp}
\end{equation}
where $W_{4}$ denotes the
Fricke involution acting on $S_{k+\frac{1}{2}}\left(\Gamma_{0}(4)\right)$,
\begin{equation}
\left(f|W_{4}\right)(z)=i^{k+\frac{1}{2}}2^{-k-\frac{1}{2}}z^{-k-\frac{1}{2}}f\left(-\frac{1}{4z}\right).\label{Fricke involution definition}
\end{equation}

\bigskip{}

Yoshida \cite{yoshida} seems to have been the first
mathematician to ever study the zeros of $L-$functions attached to these cusp forms. To illustrate
one of Yoshida's examples, let
\[
\theta(z)=\sum_{n\in\mathbb{Z}}e^{2\pi in^{2}z},\,\,\,\,\,\eta(z)=e^{\frac{\pi iz}{12}}\prod_{n=1}^{\infty}\left(1-e^{2\pi inz}\right),\,\,\,\,z\in\mathbb{H}.
\]
Then one can construct {[}\cite{shimura_half}, p. 477{]} a cusp form $g(z)\in S_{\frac{9}{2}}\left(\Gamma_{0}(4)\right)$
by taking
\begin{equation}
g(z):=\theta(z)^{-3}\eta(2z)^{12}.\label{cusp form 9/2 intro}
\end{equation}
Yoshida considered the $L-$function attached to $g$, $L(s,g)$,
and from the calculation of some of its zeros {[}\cite{yoshida}, p.
675{]}, he showed that the analogue of the Riemann hypothesis for
$L(s,g)$ is false.

Nevertheless, since there are Dirichlet series (such as the Epstein
zeta function attached to binary quadratic forms) which do not obey
to the Riemann hypothesis but still possess a positive proportion
of zeros on the critical line, it is not unreasonable to study, desspite Yoshida's
counterexample, some analogues of Hardy's Theorem for this class of $L-$funcrions. 

\bigskip{}

As far as we know, the work of J. Meher, S. Pujahari and K. Srinivas \cite{meher_srinivas}
contains the first result of Hardy-type for some $L-$functions attached to cusp forms of half-integral weight. They have proved the following theorem.

\paragraph*{Theorem A \cite{meher_srinivas}}\label{theorem A}\textit{ Suppose that $f(z)=\sum_{n=1}^{\infty}a_{f}(n)\,e^{2\pi inz}\in S_{k+\frac{1}{2}}\left(\Gamma_{0}(4)\right)$
is an eigenform for all Hecke operators $\mathcal{T}_{n^{2}}$ and for the operator
$W_{4}$ (\ref{Fricke involution definition}). Assume also that all
the Fourier coefficients of $f(z)$, $a_{f}(n)$, are either real
or purely imaginary numbers. Then the $L-$function (\ref{Dirichlet L series cusp form at intro first})
attached to $f$ has infinitely many zeros on the critical line $\text{Re}(s)=\frac{k}{2}+\frac{1}{4}$.}

\bigskip{}

This result was extended to a larger class of $L-$functions by  J. Meher, S. Pujahari and K. Shankhadhar \cite{meher_half}. For $L-$functions having twisted coefficients, a generalization of a result of Wilton was provided by a beautiful argument of Kim \cite{kim}. The author of this paper \cite{ribeiro_half} proved analogues of the Hardy-Littlewood estimate for the number of critical zeros of the $L-$functions considered by these authors. 
In this paper, however, we take the particular class of $L-$functions studied in \cite{meher_srinivas}, and use some particular features of these $L-$functions that allow to prove the uniform distribution of their zero ordinates. 

The assumptions in the Theorem A above allow to rewrite (\ref{functional equation first Cusp}) in a simpler form. Indeed, since $W_{4}$ is an involution, we know that when $f$ is an eigenform
for $W_{4}$, $f|W_{4}(z)=(-1)^{\ell}f(z)$ with $\ell\in\{0,1\}$.
Thus, when $f$ satisfies the conditions of Theorem A, the 
functional equation of its $L-$function can be written in the following form
\begin{equation}
\pi^{-s}\Gamma\left(s\right)\,L\left(s,f\right)=(-1)^{\ell}\,\pi^{-\left(k+\frac{1}{2}-s\right)}\Gamma\left(k+\frac{1}{2}-s\right)\,L\left(k+\frac{1}{2}-s,f\right).\label{functional equation iwht l}
\end{equation}

Alongside (\ref{functional equation iwht l}), other properties of the class considered by Theorem A will prove to be very important to ensure the uniform distribution of the zeros of their $L-$functions. The main goal of our paper is to prove the following result. 

\begin{theorem}\label{theorem uniform distribution} Let $f\in S_{k+\frac{1}{2}}(\Gamma_{0}(4))$ be an eigenform for all
Hecke operators $\mathcal{T}_{n^{2}}$ and the Fricke operator $W_{4}$
and let $L(s,f)$ be the associated $L-$function. Then the imaginary
parts of the nontrivial zeros of $L(s,f)$ are uniformly distributed
modulo one.
\end{theorem}

\bigskip{}

We know that $\text{dim}\,S_{\frac{9}{2}}\left(\Gamma_{0}(4)\right)=1$
{[}\cite{shimura_half}, p. 477{]}, so that $g(z)$ given by (\ref{cusp form 9/2 intro})
is an eigenform for all the Hecke operators and satisfies $\left(g|W_{4}\right)(z)=g(z)$.
Thus, the zero ordinates of $L(s,g)$ are uniformly distributed modulo one, which justifies the choice of the title for this paper. 

\begin{corollary}
Let $L(s,g)$ be the $L-$function attached to the cusp form $g(z)=\theta(z)^{-3}\eta(2z)^{12}\in S_{\frac{9}{2}}\left(\Gamma_{0}(4)\right)$.
Then the ordinates of the nontrivial zeros of $L(s,g)$ are uniformly
distributed modulo one.
\end{corollary}

\bigskip{}

The proof of our main result relies on the application of Landau's
lemma, together with an explicit proof of a density hypothesis on
the zeros of $L(s,f)$ (cf. \cite{akbary_murty}). Our paper is organized
as follows. In the next section, we revise some important lemmas about
the Fourier coefficients of the cusp forms satisfying the conditions of our Theorem \ref{theorem uniform distribution}. The most important lemma
of our next section is a result proved by Lau, Royer and Wu in \cite{lau},
which consists in an application of the Rankin-Selberg method to
study mean values of the coefficients of half-integral weight cusp
forms. Next, based on a beautiful method of Ramachandra \cite{ramachandra_montgomery,ivic,padma_thesis, Kanemitsu_Sanka}, we study the integral
\[
\intop_{0}^{T}\left|L\left(\frac{k}{2}+\frac{1}{4}+it,f\right)\right|^{2}dt
\]
as $T\rightarrow\infty$. A suitable upper bound for this integral
will prove to be essential for our establishment of a density hypothesis
for the nontrivial zeros of $L(s,f)$. Although we only employ the standard machinery, as far as we know there are no explicit results regarding mean values of $L-$functions attached to cusp forms of half-integral weight. 

Finally, we should mention that it is quite natural that our main result holds for any half-integral weight cusp form in $\Gamma_{0}(4N)$ and, most likely, the powerful method employed by M\"uller \cite{muller} could work in establishing a general analogue of our Theorem \ref{theorem uniform distribution}.\footnote{For example, M\"uller's method can be used to establish, similarly to Schmeller's result \cite{schmeller}, that the zero ordinates of the Siegel zeta functions attached to binary indefinite quadratic forms are uniformly distributed modulo one.} However, we leave this general case open for further explorations.

\section{Preliminary results}

\subsection{Properties of a subclass of cusp forms}

In this section we state and prove some important lemmas about the class of $L-$functions considered in our main Theorem. We start with the following lemma, which is given in [\cite{lau}, page 872, Proposition 7]. 

\begin{lemma}
Let $f\in S_{k+\frac{1}{2}}\left(\Gamma_{0}(4)\right)$ and assume
that $f$ is an eigenform for all the Hecke operators $\mathcal{T}_{n^{2}}$.
Then the Dirichlet series
\begin{equation}
\mathscr{D}\left(f\otimes\overline{f},s\right):=\sum_{n=1}^{\infty}\frac{|a_{f}(n)|^{2}}{n^{s}}
\end{equation}
converges absolutely for $\text{Re}(s)>k+\frac{1}{2}$. Moreover,
it can be analytically continued in the half-plane $\text{Re}(s)>k$ as a meromorphic function, 
with its only (simple) pole being located at $s=k+\frac{1}{2}$.
\end{lemma}

A consequence of the previous lemma, combined with a strong result
of Conrey and Iwaniec \cite{conrey_iwaniec} (cf.  {[}\cite{lau}, p. 877, eq. (14){]}), is that the
mean value estimate takes place
\begin{equation}
\sum_{n\leq x}|a_{f}(n)|^{2}=\frac{2r_{f}}{2k+1}\,x^{k+\frac{1}{2}}+O\left(x^{k+\frac{5}{12}+\epsilon}\right),\label{mean discrete}
\end{equation}
where $r_{f}$ is the residue of $\mathscr{D}\left(f\otimes\overline{f},s\right)$
at $s=k+\frac{1}{2}$. Of course, for the purposes of our Theorem \ref{theorem uniform distribution},
it is enough that $\sum_{n\leq x}|a_{f}(n)|^{2}\ll x^{k+\frac{1}{2}}$.   

\bigskip{}

When combined with the Cauchy-Schwarz inequality (cf. {[}\cite{meher_srinivas},
p. 929, Proposition 2.3{]}), the previous lemma establishes that,
for any $f\in S_{k+\frac{1}{2}}\left(\Gamma_{0}(4)\right)$ which
is an eigenform for all Hecke operators $\mathcal{T}_{n^{2}}$, the
$L-$function
\[
L(s,f)=\sum_{n=1}^{\infty}\frac{a_{f}(n)}{n^{s}},
\]
converges absolutely when $\text{Re}(s)>\frac{k}{2}+\frac{3}{4}$.

\bigskip{}

By following a method of Ramachandra \cite{ivic,ramachandra_montgomery},
the next couple of lemmas are devoted to prove a mean value estimate for $L\left(s,f\right)$
on the critical line $\text{Re}(s)=\frac{k}{2}+\frac{1}{4}$. We start with a result that allows to write an approximate functional equation for $L(s,f)$ that is helpful in the study of such quantities. 

\begin{lemma}\label{approximate functional equation}

Let $f\in S_{k+\frac{1}{2}}(\Gamma_{0}(4))$ be an eigenform for all
Hecke operators $\mathcal{T}_{n^{2}}$ and the Fricke operator $W_{4}$, i.e.,  $\left(f|W_{4}\right)(z)=(-1)^{\ell}f(z), \ell=0,1.$

Write $s=\sigma+it$ and assume that $\frac{k}{2}+\frac{1}{4}\leq\sigma\leq\frac{k}{2}+\frac{3}{4}$.
Then the following approximate functional equation holds
\begin{align}
L(s,f) & =\sum_{n\leq x}\frac{a_{f}(n)}{n^{s}}+\sum_{n\leq x}\frac{a_{f}(n)}{n^{s}}\left(e^{-\left(n/x\right)^{h}}-1\right)\nonumber \\
 & +\sum_{n>x}\frac{a_{f}(n)}{n^{s}}e^{-\left(n/x\right)^{h}}+(-1)^{\ell}\pi^{2s-k-\frac{1}{2}}\frac{\Gamma\left(k+\frac{1}{2}-s\right)}{\Gamma(s)}\,\sum_{n\leq x}\frac{a_{f}(n)}{n^{k+\frac{1}{2}-s}}\nonumber \\
 & -\frac{(-1)^{\ell}\pi^{2s-k-\frac{1}{2}}}{2\pi i\,h}\intop_{-\eta-i\infty}^{-\eta+i\infty}\frac{\Gamma\left(k+\frac{1}{2}-s-w\right)\Gamma\left(\frac{w}{h}\right)}{\Gamma\left(s+w\right)}\sum_{n>x}\frac{a_{f}(n)}{n^{k+\frac{1}{2}-s-w}}\left(\pi^{2}x\right)^{w}dw\nonumber \\
 & -\frac{(-1)^{\ell}\pi^{2s-k-\frac{1}{2}}}{2\pi i\,h}\intop_{\alpha-i\infty}^{\alpha+i\infty}\frac{\Gamma\left(k+\frac{1}{2}-s-w\right)\Gamma\left(\frac{w}{h}\right)}{\Gamma\left(s+w\right)}\sum_{n\leq x}\frac{a_{f}(n)}{n^{k+\frac{1}{2}-s-w}}\left(\pi^{2}x\right)^{w}dw,\label{Reflection principle}
\end{align}
for $0<\alpha<\frac{k}{2}+\frac{1}{4}$, $h=1,2$ and $\eta>\sigma-\frac{k}{2}+\frac{1}{4}$.
Moreover, one has the truncated formula
\begin{align}
L(s,f) & =\sum_{n\leq x}\frac{a_{f}(n)}{n^{s}}+\sum_{n\leq x}\frac{a_{f}(n)}{n^{s}}\left(e^{-\left(n/x\right)^{h}}-1\right)\nonumber \\
 & +\sum_{n>x}\frac{a_{f}(n)}{n^{s}}e^{-\left(n/x\right)^{h}}+(-1)^{\ell}\pi^{2s-k-\frac{1}{2}}\frac{\Gamma\left(k+\frac{1}{2}-s\right)}{\Gamma(s)}\,\sum_{n\leq x}\frac{a_{f}(n)}{n^{k+\frac{1}{2}-s}}\nonumber \\
 & -\frac{(-1)^{\ell}\pi^{2s-k-\frac{1}{2}}}{2\pi i\,h}\intop_{-\eta-i\log^{2}|t|}^{-\eta+i\log^{2}|t|}\frac{\Gamma\left(k+\frac{1}{2}-s-w\right)\Gamma\left(\frac{w}{h}\right)}{\Gamma\left(s+w\right)}\sum_{n>x}\frac{a_{f}(n)}{n^{k+\frac{1}{2}-s-w}}\left(\pi^{2}x\right)^{w}dw\nonumber \\
 & -\frac{(-1)^{\ell}\pi^{2s-k-\frac{1}{2}}}{2\pi i\,h}\intop_{\alpha-i\log^{2}|t|}^{\alpha+i\log^{2}|t|}\frac{\Gamma\left(k+\frac{1}{2}-s-w\right)\Gamma\left(\frac{w}{h}\right)}{\Gamma\left(s+w\right)}\sum_{n\leq x}\frac{a_{f}(n)}{n^{k+\frac{1}{2}-s-w}}\left(\pi^{2}x\right)^{w}dw\nonumber \\
 & +O\left(|t|^{-A}\right),\label{Truncated Reflection principle}
\end{align}
where $A$ is an arbitrarily large positive constant.
\end{lemma}

\begin{proof}
Let $c>\frac{k}{2}+\frac{3}{4}$ and use the Cahen-Mellin representation
to write
\begin{align}
\sum_{n=1}^{\infty}\frac{a_{f}(n)}{n^{s}}e^{-\left(n/x\right)^{h}} & =\frac{1}{h}\,\sum_{n=1}^{\infty}\frac{a_{f}(n)}{n^{s}}\,\frac{1}{2\pi i}\intop_{c-i\infty}^{c+i\infty}\Gamma\left(\frac{w}{h}\right)\left(\frac{n}{x}\right)^{-w}dw\nonumber \\
 & =\frac{1}{2\pi ih}\,\intop_{c-i\infty}^{c+i\infty}\Gamma\left(\frac{w}{h}\right)L\left(s+w,f\right)\,x^{w}dw,\label{af(n)ns}
\end{align}
where the interchange of the orders of summation and integration is
justified via absolute convergence (since $c>\frac{k}{2}+\frac{3}{4}$
by hypothesis). We now move the line of integration to $\text{Re}(w)=-\eta$
and integrate along a positively oriented rectangular contour whose
vertices are $\left[c\pm iT,-\eta\pm iT\right]$ and then let $T\rightarrow\infty$.
By the choice of $\eta$, we know that the integrand has a simple
pole located at $w=0$, whose residue is $h\,L(s,f)$. Hence,
\begin{equation}
\sum_{n=1}^{\infty}\frac{a_{f}(n)}{n^{s}}e^{-\left(n/x\right)^{h}}=\frac{1}{2\pi i\,h}\intop_{-\eta-i\infty}^{-\eta+i\infty}L(s+w,f)\,\Gamma\left(\frac{w}{h}\right)\,x^{w}dw+L(s,f).\label{equation to prove at beginning}
\end{equation}
We now use the functional equation for $L(s+w,f)$ and the fact that
$\eta>\sigma-\frac{k}{2}+\frac{1}{4}$, to write the integral along $\text{Re}=-\eta$ in the form
\begin{align}
\frac{1}{2\pi i\,h}\intop_{-\eta-i\infty}^{-\eta+i\infty}L(s+w,f)\,\Gamma\left(\frac{w}{h}\right)\,x^{w}dw\nonumber \\
=\frac{(-1)^{\ell}\pi^{2s-k-\frac{1}{2}}}{2\pi i\,h}\intop_{-\eta-i\infty}^{-\eta+i\infty}\frac{\Gamma\left(\frac{w}{h}\right)\Gamma\left(k+\frac{1}{2}-s-w\right)}{\Gamma\left(s+w\right)}L\left(k+\frac{1}{2}-s-w,f\right)\,\left(\pi^{2}x\right)^{w}dw\nonumber \\
=\frac{(-1)^{\ell}\pi^{2s-k-\frac{1}{2}}}{2\pi i\,h}\intop_{-\eta-i\infty}^{-\eta+i\infty}\frac{\Gamma\left(\frac{w}{h}\right)\Gamma\left(k+\frac{1}{2}-s-w\right)}{\Gamma\left(s+w\right)}\sum_{n=1}^{\infty}\frac{a_{f}(n)}{n^{k+\frac{1}{2}-s-w}}\,\left(\pi^{2}x\right)^{w}dw.\label{reqriting stuuuf}
\end{align}
Next, we truncate the Dirichlet series associated with $L\left(k+\frac{1}{2}-s-w,f\right)$
at the point $x$, which gives the expression
\begin{align*}
\frac{1}{2\pi i\,h}\intop_{-\eta-i\infty}^{-\eta+i\infty}L(s+w,f)\,\Gamma\left(\frac{w}{h}\right)\,x^{w}dw & =\frac{(-1)^{\ell}\pi^{2s-k-\frac{1}{2}}}{2\pi i\,h}\intop_{-\eta-i\infty}^{-\eta+i\infty}\frac{\Gamma\left(\frac{w}{h}\right)\Gamma\left(k+\frac{1}{2}-s-w\right)}{\Gamma\left(s+w\right)}\sum_{n\leq x}\frac{a_{f}(n)}{n^{k+\frac{1}{2}-s-w}}\,\left(\pi^{2}x\right)^{w}dw\\
+\frac{(-1)^{\ell}\pi^{2s-k-\frac{1}{2}}}{2\pi i\,h} & \intop_{-\eta-i\infty}^{-\eta+i\infty}\frac{\Gamma\left(\frac{w}{h}\right)\Gamma\left(k+\frac{1}{2}-s-w\right)}{\Gamma\left(s+w\right)}\sum_{n>x}\frac{a_{f}(n)}{n^{k+\frac{1}{2}-s-w}}\,\left(\pi^{2}x\right)^{w}dw.
\end{align*}

We move once more the line of integration on the first integral from
$\text{Re}(w)=-\eta$ to $\text{Re}(w)=\alpha$, with $0<\alpha<\frac{k}{2}+\frac{1}{4}$.
Doing so, and taking once more into consideration the pole that $\Gamma\left(w/h\right)$
possesses at $w=0$, we find the expression
\begin{align}
\frac{(-1)^{\ell}\pi^{2s-k-\frac{1}{2}}}{2\pi i\,h}\intop_{-\eta-i\infty}^{-\eta+i\infty}\frac{\Gamma\left(\frac{w}{h}\right)\Gamma\left(k+\frac{1}{2}-s-w\right)}{\Gamma\left(s+w\right)}\sum_{n\leq x}\frac{a_{f}(n)}{n^{k+\frac{1}{2}-s-w}}\,\left(\pi^{2}x\right)^{w}dw\nonumber \\
=\frac{(-1)^{\ell}\pi^{2s-k-\frac{1}{2}}}{2\pi i\,h}\intop_{\alpha-i\infty}^{\alpha+i\infty}\frac{\Gamma\left(\frac{w}{h}\right)\Gamma\left(k+\frac{1}{2}-s-w\right)}{\Gamma\left(s+w\right)}\sum_{n\leq x}\frac{a_{f}(n)}{n^{k+\frac{1}{2}-s-w}}\,\left(\pi^{2}x\right)^{w}dw\nonumber \\
+(-1)^{\ell-1}\pi^{2s-k-\frac{1}{2}}\,\frac{\Gamma\left(k+\frac{1}{2}-s\right)}{\Gamma(s)}\sum_{n\leq x}\frac{a_{f}(n)}{n^{k+\frac{1}{2}-s}}.\label{part after shifing to alpha}
\end{align}

Returning to (\ref{equation to prove at beginning}), (\ref{reqriting stuuuf})
and, finally, to (\ref{part after shifing to alpha}) we are able to write the following approximate functional equation
\begin{align*}
L(s,f) & =\sum_{n=1}^{\infty}\frac{a_{f}(n)}{n^{s}}e^{-\left(n/x\right)^{h}}+(-1)^{\ell}\pi^{2s-k-\frac{1}{2}}\frac{\Gamma\left(k+\frac{1}{2}-s\right)}{\Gamma(s)}\sum_{n\leq x}\frac{a_{f}(n)}{n^{k+\frac{1}{2}-s}}\\
 & -\frac{(-1)^{\ell}\pi^{2s-k-\frac{1}{2}}}{2\pi i\,h}\intop_{-\eta-i\infty}^{-\eta+i\infty}\frac{\Gamma\left(\frac{w}{h}\right)\Gamma\left(k+\frac{1}{2}-s-w\right)}{\Gamma\left(s+w\right)}\sum_{n>x}\frac{a_{f}(n)}{n^{k+\frac{1}{2}-s-w}}\,\left(\pi^{2}x\right)^{w}dw\\
 & -\frac{(-1)^{\ell}\pi^{2s-k-\frac{1}{2}}}{2\pi i\,h}\intop_{\alpha-i\infty}^{\alpha+i\infty}\frac{\Gamma\left(\frac{w}{h}\right)\Gamma\left(k+\frac{1}{2}-s-w\right)}{\Gamma\left(s+w\right)}\sum_{n\leq x}\frac{a_{f}(n)}{n^{k+\frac{1}{2}-s-w}}\,\left(\pi^{2}x\right)^{w}dw.
\end{align*}
At last, if we note that
\[
\sum_{n=1}^{\infty}\frac{a_{f}(n)}{n^{s}}e^{-(n/x)^{h}}=\sum_{n\leq x}\frac{a_{f}(n)}{n^{s}}+\sum_{n\leq x}\frac{a_{f}(n)}{n^{s}}\left(e^{-(n/x)^{h}}-1\right)+\sum_{n>x}\frac{a_{f}(n)}{n^{s}}e^{-(n/x)^{h}},
\]
then (\ref{Reflection principle})
follows. The derivation of the truncated version (\ref{Truncated Reflection principle})
is analogous: starting with the Cahen-Mellin integral (\ref{af(n)ns}),
we only truncate the line of integration there at the point $|\text{Im}(w)|=\log^{2}|t|$.
By Stirling's formula, the integral along the half lines $|\text{Im}(w)|>\log^{2}|t|$
can be estimated as $O(|t|^{-A})$ for any large positive constant
$A$. This shows the formula
\[
\sum_{n=1}^{\infty}\frac{a_{f}(n)}{n^{s}}e^{-\left(n/x\right)^{h}}=\frac{1}{2\pi ih}\,\intop_{c-i\log^{2}|t|}^{c+i\log^{2}|t|}\Gamma\left(\frac{w}{h}\right)L\left(s+w,f\right)\,x^{w}dw+O\left(|t|^{-A}\right),
\]
from which one deduces (\ref{Truncated Reflection principle}) in
the same way as one obtains (\ref{Reflection principle}). 
\end{proof}

\subsection{A mean value theorem for $L(s,f)$}

Before proving the main lemma of this section, we need to recall a result due to Montgomery and Vaughan \cite{montgomery_vaughan} (see also \cite{ramachandra_note} for a simpler proof of this result). 

\begin{lemma}\label{Lemma Montgomery Vaughan}
For $n=1,...,N$, let $a_{n}$ and $b_{n}$ be two complex numbers. Then we have that 
\begin{equation}
\intop_{0}^{T}\left(\sum_{n=1}^{N}a_{n}n^{-it}\right)\left(\sum_{n=1}^{N}b_{n}n^{it}\right)dt=T\,\sum_{n=1}^{M}a_{n}b_{n}+O\left(\left(\sum_{n=1}^{N}n|a_{n}|^{2}\right)^{1/2}\left(\sum_{n=1}^{N}n|b_{n}|^{2}\right)^{1/2}\right). \label{Montgomery_Vaughan identity}
\end{equation}
\end{lemma}

\bigskip{}

Appealing to the previous approximate functional equation (\ref{Truncated Reflection principle}), we use (\ref{Montgomery_Vaughan identity}) to find a mean value estimate for $L(s,f)$
on the critical line $\text{Re}(s)=\frac{k}{2}+\frac{1}{4}$. This is the content of the next lemma. 

\begin{lemma} \label{ramachandra mean}
Let $f\in S_{k+\frac{1}{2}}(\Gamma_{0}(4))$ be an eigenform for all
Hecke operators $\mathcal{T}_{n^{2}}$ and the Fricke operator $W_{4}$
and let $L(s,f)$ be the associated $L-$function. Then we have that
\begin{equation}
\intop_{0}^{T}\left|L\left(\frac{k}{2}+\frac{1}{4}+it,f\right)\right|^{2}dt=A_{f}T\log(T)+O(T),\label{mean value cusp form half}
\end{equation}
where $A_{f}$ is a constant depending on the cusp form $f$.
\end{lemma}

\begin{proof}
Let $\frac{k}{2}+\frac{1}{4}\leq\text{Re}(s)=\sigma\leq\frac{k}{2}+\frac{3}{4}$, $\text{Im}(s)=t$, $h=1$ and $x=T$,
for $\frac{T}{2}\leq t\leq T$. Since $t\asymp T$, (\ref{Truncated Reflection principle})
gives the approximate formula
\begin{align}
L(s,f) & =\sum_{n\leq T}\frac{a_{f}(n)}{n^{s}}+\sum_{n\leq T}\frac{a_{f}(n)}{n^{s}}\left(e^{-\frac{n}{T}}-1\right)\nonumber \\
 & +\sum_{n>T}\frac{a_{f}(n)}{n^{s}}e^{-\frac{n}{T}}+(-1)^{\ell}\pi^{2s-k-\frac{1}{2}}\frac{\Gamma\left(k+\frac{1}{2}-s\right)}{\Gamma(s)}\,\sum_{n\leq T}\frac{a_{f}(n)}{n^{k+\frac{1}{2}-s}}\nonumber \\
 & -\frac{(-1)^{\ell}\pi^{2s-k-\frac{1}{2}}}{2\pi i}\intop_{-\eta-i\log^{2}|t|}^{-\eta+i\log^{2}|t|}\frac{\Gamma\left(k+\frac{1}{2}-s-w\right)\Gamma\left(\frac{w}{h}\right)}{\Gamma\left(s+w\right)}\sum_{n>T}\frac{a_{f}(n)}{n^{k+\frac{1}{2}-s-w}}\left(\pi^{2}T\right)^{w}dw\nonumber \\
 & -\frac{(-1)^{\ell}\pi^{2s-k-\frac{1}{2}}}{2\pi i}\intop_{\alpha-i\log^{2}|t|}^{\alpha+i\log^{2}|t|}\frac{\Gamma\left(k+\frac{1}{2}-s-w\right)\Gamma\left(\frac{w}{h}\right)}{\Gamma\left(s+w\right)}\sum_{n\leq T}\frac{a_{f}(n)}{n^{k+\frac{1}{2}-s-w}}\left(\pi^{2}T\right)^{w}dw+O\left(T^{-A}\right)\nonumber \\
 & :=\sum_{\ell=1}^{6}\mathcal{J}_{\ell}(s)+O(T^{-A}),\,\,\,\,\,\,\text{say},\label{More general gormula!}
\end{align}
where $\eta>\sigma-\frac{k}{2}+\frac{1}{4}$ and $0<\alpha<\frac{k}{2}+\frac{1}{4}$. 
Furthermore, if we replace $s=\frac{k}{2}+\frac{1}{4}+it$ and set
$J_{\ell}(t):=\mathcal{J}_{\ell}\left(\frac{k}{2}+\frac{1}{4}+it\right)$,
the previous formula implies
\begin{align}
L\left(\frac{k}{2}+\frac{1}{4}+it,f\right) & =\sum_{n\leq T}\frac{a_{f}(n)}{n^{\frac{k}{2}+\frac{1}{4}+it}}+\sum_{n\leq T}\frac{a_{f}(n)}{n^{\frac{k}{2}+\frac{1}{4}+it}}\left(e^{-\frac{n}{T}}-1\right)\nonumber \\
 & +\sum_{n>T}\frac{a_{f}(n)}{n^{\frac{k}{2}+\frac{1}{4}+it}}e^{-\frac{n}{T}}+(-1)^{\ell}\pi^{2it}\frac{\Gamma\left(\frac{k}{2}+\frac{1}{4}-it\right)}{\Gamma\left(\frac{k}{2}+\frac{1}{4}+it\right)}\,\sum_{n\leq T}\frac{a_{f}(n)}{n^{\frac{k}{2}+\frac{1}{4}-it}}\nonumber \\
 & -\frac{(-1)^{\ell}\pi^{2it}}{2\pi i}\,\intop_{-\eta-i\log^{2}T}^{-\eta+i\log^{2}T}\frac{\Gamma\left(\frac{k}{2}+\frac{1}{4}-it-w\right)\Gamma\left(w\right)}{\Gamma\left(\frac{k}{2}+\frac{1}{4}+it+w\right)}\sum_{n>T}\frac{a_{f}(n)}{n^{\frac{k}{2}+\frac{1}{4}-it-w}}\left(\pi^{2}T\right)^{w}dw\nonumber \\
 & -\frac{(-1)^{\ell}\pi^{2it}}{2\pi i}\,\intop_{\alpha-i\log^{2}T}^{\alpha+i\log^{2}T}\frac{\Gamma\left(\frac{k}{2}+\frac{1}{4}-it-w\right)\Gamma\left(w\right)}{\Gamma\left(\frac{k}{2}+\frac{1}{4}+it+w\right)}\sum_{n\leq T}\frac{a_{f}(n)}{n^{\frac{k}{2}+\frac{1}{4}-it-w}}\left(\pi^{2}T\right)^{w}dw+O\left(T^{-A}\right)\nonumber \\
 & :=\sum_{\ell=1}^{6}J_{\ell}(t)+O(T^{-A}),\label{writing as 6 integrals!}
\end{align}
where, by the conditions of Lemma \ref{approximate functional equation}, $\eta>\frac{1}{2}$.
Note that $|J_{1}(t)|=|J_{4}(t)|$: we use the previous decomposition
to write the mean square in the form
\begin{equation}
\intop_{T/2}^{T}\left|L\left(\frac{k}{2}+\frac{1}{4}+it,f\right)\right|^{2}dt=\sum_{\ell=1}^{6}\intop_{\frac{T}{2}}^{T}|J_{\ell}(t)|^{2}dt+O\left(\sum_{m\neq\ell}\intop_{\frac{T}{2}}^{T}J_{m}(t)\,\overline{J_{\ell}(t)}\,dt\right).\label{after approximate functional in the proof}
\end{equation}

The proof of this result is a combination of three claims. For simplicity
of the exposition, our first claim essentially consists in evaluating
the first six terms of the equality (\ref{after approximate functional in the proof}).

\begin{claim}\label{claim 2.1}
The following estimates hold
\begin{equation}
\intop_{\frac{T}{2}}^{T}|J_{1}(t)|^{2}dt=\intop_{\frac{T}{2}}^{T}|J_{4}(t)|^{2}dt=\frac{r_{f}}{2}T\log(T)+O(T),\label{J1 J4 claim}
\end{equation}
\begin{equation}
\intop_{\frac{T}{2}}^{T}\left|J_{2}(t)\right|^{2}dt\ll T,\label{J2 claim}
\end{equation}
\begin{equation}
\intop_{\frac{T}{2}}^{T}|J_{3}(t)|^{2}dt\ll1,\label{J3 claim}
\end{equation}
\begin{equation}
\intop_{T/2}^{T}|J_{5}(t)|^{2}dt\ll T,\label{J5 square claim}
\end{equation}
and, finally,
\begin{equation}
\intop_{T/2}^{T}|J_{6}(t)|^{2}dt\ll T,\label{J6 square claim}
\end{equation}
where in (\ref{J1 J4 claim}) $r_{f}$ denotes the constant appearing on the main term of (\ref{mean discrete}). 
\end{claim}

\begin{proof}
Since $|J_{1}(t)|=|J_{4}(t)|$, we have that
\begin{equation}
\intop_{\frac{T}{2}}^{T}|J_{1}(t)|^{2}dt=\intop_{\frac{T}{2}}^{T}|J_{4}(t)|^{2}dt,\label{J1 equal J4}
\end{equation}
so, if we prove that the right-hand side of (\ref{J1 J4 claim}) equals
to one of the sides of (\ref{J1 equal J4}), then (\ref{J1 J4 claim})
is completely proved. By (\ref{Montgomery_Vaughan identity}) and (\ref{mean discrete}), 
\begin{align*}
\intop_{\frac{T}{2}}^{T}|J_{1}(t)|^{2}dt & =\intop_{\frac{T}{2}}^{T}\left|\sum_{n\leq T}\frac{a_{f}(n)}{n^{\frac{k}{2}+\frac{1}{4}+it}}\right|^{2}\,dt=\sum_{n\leq T}\frac{|a_{f}(n)|^{2}}{n^{k+\frac{1}{2}}}\left(\frac{T}{2}+O\left(n\right)\right)\\
 & =\frac{1}{2}r_{f}T\log(T)+\frac{1}{2}r_{f}^{(0)}T+O\left(T^{\frac{5}{12}+\epsilon}\right)+O\left(\sum_{n\leq T}\frac{|a_{f}(n)|^{2}}{n^{k-\frac{1}{2}}}\right)\\
 & =\frac{1}{2}r_{f}T\log(T)+\frac{1}{2}r_{f}^{(0)}T+O\left(T^{\frac{5}{12}+\epsilon}\right)+O\left(r_{f}T\right)+O\left(T^{\frac{11}{12}+\epsilon}\right)\\
 & =\frac{r_{f}}{2}T\log(T)+O(T),
\end{align*}
and so we have that
\[
\intop_{\frac{T}{2}}^{T}|J_{4}(t)|^{2}dt=\intop_{\frac{T}{2}}^{T}\left|J_{1}(t)\right|^{2}dt=\frac{r_{f}}{2}T\log(T)+O(T).
\]
Next, the estimates of the mean values of $J_{2}(t)$ and $J_{3}(t)$ can be easily obtained in the following form, 
\begin{align*}
\intop_{\frac{T}{2}}^{T}\left|J_{2}(t)\right|^{2}dt & =\intop_{\frac{T}{2}}^{T}\left|\sum_{n\leq T}\frac{a_{f}(n)}{n^{\frac{k}{2}+\frac{1}{4}+it}}\left(e^{-\frac{n}{T}}-1\right)\right|^{2}dt\\
 & \ll T\sum_{n\leq T}\frac{|a_{f}(n)|^{2}}{n^{k+\frac{1}{2}}}\left(e^{-\frac{n}{T}}-1\right)^{2}+\sum_{n\leq T}\frac{|a_{f}(n)|^{2}}{n^{k-\frac{1}{2}}}\left(e^{-\frac{n}{T}}-1\right)^{2}\\
 & \ll\frac{1}{T}\,\sum_{n\leq T}\frac{|a_{f}(n)|^{2}}{n^{k-\frac{3}{2}}}+\frac{1}{T^{2}}\sum_{n\leq T}\frac{|a_{f}(n)|^{2}}{n^{k-\frac{5}{2}}}\\
 & \ll T
\end{align*}
and
\begin{align*}
\intop_{\frac{T}{2}}^{T}|J_{3}(t)|^{2}dt & =\intop_{\frac{T}{2}}^{T}\left|\sum_{n>T}\frac{a_{f}(n)}{n^{\frac{k}{2}+\frac{1}{4}+it}}e^{-\frac{n}{T}}\right|^{2}dt=\sum_{n>T}\frac{|a_{f}(n)|^{2}}{n^{k+\frac{1}{2}}}e^{-\frac{2n}{T}}\left(\frac{T}{2}+O(n)\right)\\
 & \ll1.
\end{align*}

We now bound the mean values of $J_{5}(t)$ and $J_{6}(t)$: by the Cauchy-Schwarz
inequality and Stirling's formula, one is able to deduce that
\begin{align}
\intop_{T/2}^{T}|J_{5}(t)|^{2}dt & \ll\intop_{T/2}^{T}\intop_{-\eta-i\log^{2}T}^{-\eta+i\log^{2}T}\left|\frac{\Gamma\left(\frac{k}{2}+\frac{1}{4}-it-w\right)\Gamma\left(w\right)}{\Gamma\left(\frac{k}{2}+\frac{1}{4}+it+w\right)}(\pi^{2}T)^{w}\right|^{2}\,\left|\sum_{n>T}\frac{a_{f}(n)}{n^{\frac{k}{2}+\frac{1}{4}-it-w}}\right|^{2}\,|dw|\,dt\nonumber \\
 & \ll T^{2\eta}\,\intop_{-\eta-i\log^{2}(T)}^{-\eta+i\log^{2}(T)}\left|\Gamma(w)\right|^{2}\,\intop_{T/2}^{T}\left|\sum_{n>T}\frac{a_{f}(n)}{n^{\frac{k}{2}+\frac{1}{4}-it-w}}\right|^{2}dt\,|dw|,\label{J5(t)2}
\end{align}
where the interchange of the orders of integration is justified due
to the fact that the integrals are over compact segments and the choice
of $\eta>\frac{1}{2}$ making the series absolutely convergent. We now appeal
to Montgomery-Vaughan's result (\ref{Montgomery_Vaughan identity}) to estimate the integral with
respect to $t$: since, by Stirling's formula, $\Gamma(-\eta+it)\in L_{1}(\mathbb{R})$
for $\eta\notin\mathbb{N}_{0}$, we have that 
\begin{align}
\intop_{T/2}^{T}|J_{5}(t)|^{2}dt & \ll T^{2\eta}\,\intop_{-\eta-i\log^{2}(T)}^{-\eta+i\log^{2}(T)}\left|\Gamma(w)\right|^{2}\,\sum_{n>T}\frac{|a_{f}(n)|^{2}}{n^{k+2\eta+\frac{1}{2}}}\left(T+n\right)\,|dw|\nonumber \\
 & \ll T^{2\eta}\,\sum_{n>T}\frac{|a_{f}(n)|^{2}}{n^{k+2\eta+\frac{1}{2}}}\left(T+n\right)\nonumber \\
 & =T^{1+2\eta}\intop_{T}^{\infty}t^{-k-\frac{1}{2}-2\eta}d\left(\sum_{T<n\leq t}|a_{f}(n)|^{2}\right)\nonumber \\
 & +T^{2\eta}\intop_{T}^{\infty}t^{-k+\frac{1}{2}-2\eta}d\left(\sum_{T<n\leq t}|a_{f}(n)|^{2}\right),\label{at most combining}
\end{align}
where we have used summation by parts. Let us estimate the first integral
(the second integral is analogous): an integration by parts yields
\begin{align}
\intop_{T}^{\infty}t^{-k-\frac{1}{2}-2\eta}d\left(\sum_{T<n\leq t}|a_{f}(n)|^{2}\right) & =\lim_{\tau\rightarrow\infty}\tau^{-k-\frac{1}{2}-2\eta}\sum_{T<n\leq\tau}|a_{f}(n)|^{2}\nonumber \\
 & +\left(k+\frac{1}{2}+2\eta\right)\,\intop_{T}^{\infty}\sum_{T<n\leq t}|a_{f}(n)|^{2}\,t^{-k-\frac{3}{2}-2\eta}dt\nonumber \\
 & \ll T^{-2\eta},\label{estimate first guy}
\end{align}
because $\sum_{n\leq\tau}|a_{f}(n)|^{2}\ll\tau^{k+\frac{1}{2}}$ (see
(\ref{mean discrete}) above). The second integral can be estimated
analogously and the same procedure holds because $\eta>\frac{1}{2}$
by hypothesis. In fact, one can deduce that
\begin{equation}
\intop_{T}^{\infty}t^{-k+\frac{1}{2}-2\eta}d\left(\sum_{T<n\leq t}|a_{f}(n)|^{2}\right)\ll T^{1-2\eta},\label{estimate simple second guy}
\end{equation}
A combination of (\ref{at most combining}) and (\ref{estimate first guy})
shows that
\begin{equation}
\intop_{T/2}^{T}|J_{5}(t)|^{2}dt\ll T.\label{J5(t) estimate}
\end{equation}

We now estimate $\intop_{T/2}^{T}|J_{6}(t)|^{2}dt$: the proof is
exactly the same as the estimate of (\ref{J5(t) estimate}), so we
shall be brief. Indeed, the Cauchy-Schwarz inequality
\begin{align*}
\intop_{T/2}^{T}|J_{6}(t)|^{2}dt & \ll\intop_{T/2}^{T}\intop_{\alpha-i\log^{2}T}^{\alpha+i\log^{2}T}\left|\frac{\Gamma\left(\frac{k}{2}+\frac{1}{4}-it-w\right)\Gamma\left(w\right)}{\Gamma\left(\frac{k}{2}+\frac{1}{4}+it+w\right)}\left(\pi^{2}T\right)^{w}\right|^{2}\,\left|\sum_{n\leq T}\frac{a_{f}(n)}{n^{\frac{k}{2}+\frac{1}{4}-it-w}}\right|^{2}|dw|\,dt\\
 & \ll T^{-2\alpha}\intop_{\alpha-i\log^{2}T}^{\alpha+i\log^{2}T}|\Gamma(w)|^{2}\,\intop_{T/2}^{T}\left|\sum_{n\leq T}\frac{a_{f}(n)}{n^{\frac{k}{2}+\frac{1}{4}-it-w}}\right|^{2}dt\,|dw|\\
 & \ll T^{-2\alpha}\,\sum_{n\leq T}\frac{|a_{f}(n)|^{2}}{n^{k+\frac{1}{2}-2\alpha}}\left(T+O(n)\right)\\
 & \ll T,
\end{align*}
where we have used (\ref{mean discrete}) and the fact that $0<\alpha<\frac{k}{2}+\frac{1}{4}$.
\end{proof}
\bigskip{}

We have dealt with mean value estimates for $J_{\ell}(t)$. The next
claim does an analogous job for the mean value estimates of $\mathcal{J}_{\ell}(\sigma+it)$,
$\sigma\neq\frac{k}{2}+\frac{1}{4}$.

\begin{claim} \label{claim 2.2}

The following estimates take place
\begin{equation}
\intop_{\frac{T}{2}}^{T}\left|\mathcal{J}_{1}\left(\sigma+it\right)\right|^{2}\,dt\ll T^{k+\frac{3}{2}-2\sigma},\,\,\,\,0<\sigma<\frac{k}{2}+\frac{1}{4},\label{1st estimate claim 2}
\end{equation}
\begin{equation}
\intop_{\frac{T}{2}}^{T}\left|\mathcal{J}_{2}\left(\sigma+it\right)\right|^{2}\,dt\ll T^{k+\frac{3}{2}-2\sigma},\,\,\,\,0<\sigma<\frac{k}{2}+\frac{3}{4},\label{first estimate claim 3}
\end{equation}
\begin{equation}
\intop_{\frac{T}{2}}^{T}\left|\mathcal{J}_{3}\left(\sigma+it\right)\right|^{2}dt\ll1,\,\,\,\,0<\sigma<\frac{k}{2}+\frac{3}{4},\label{third estimate outside line}
\end{equation}
\begin{equation}
\intop_{\frac{T}{2}}^{T}\left|\mathcal{J}_{4}\left(\sigma+it\right)\right|^{2}dt\ll T^{k+\frac{3}{2}-2\sigma},\,\,\,\,\frac{k}{2}+\frac{1}{4}<\sigma<\frac{k}{2}+\frac{3}{4},\label{second estimate claim 2}
\end{equation}
and
\begin{equation}
\intop_{\frac{T}{2}}^{T}\left|\mathcal{J}_{5}\left(\sigma+it\right)\right|^{2}\,dt\ll T^{k+\frac{3}{2}-2\sigma},\,\,\,\,0<\sigma<\frac{k}{2}+\frac{3}{4},\label{J5 estimate outside strip}
\end{equation}
\begin{equation}
\intop_{\frac{T}{2}}^{T}\left|\mathcal{J}_{6}\left(\sigma+it\right)\right|^{2}\,dt\ll T^{k+\frac{3}{2}-2\sigma},\,\,\,\,0<\sigma<\frac{k}{2}+\frac{3}{4}.\label{J6 estimate outside strip line}
\end{equation}

\end{claim}

\begin{proof}
The first proof consists in using Lemma \ref{Lemma Montgomery Vaughan}
and (\ref{mean discrete}): indeed,
\begin{align}
\intop_{\frac{T}{2}}^{T}\left|\sum_{n\leq T}\frac{a_{f}(n)}{n^{\sigma+it}}\right|^{2}\,dt & =\intop_{\frac{T}{2}}^{T}\left|\mathcal{J}_{1}\left(\sigma+it\right)\right|^{2}\,dt=\sum_{n\leq T}\frac{|a_{f}(n)|^{2}}{n^{2\sigma}}\left(\frac{T}{2}+O(n)\right)\nonumber \\
 & \ll T^{k+\frac{3}{2}-2\sigma}.\label{After using mean and montgomery}
\end{align}
The second estimate can be found by the following inequalities
\begin{align*}
\intop_{\frac{T}{2}}^{T}\left|\mathcal{J}_{2}\left(\sigma+it\right)\right|^{2}dt & =\intop_{\frac{T}{2}}^{T}\left|\sum_{n\leq T}\frac{a_{f}(n)}{n^{\sigma+it}}\left(e^{-\frac{n}{T}}-1\right)\right|^{2}dt\\
 & \ll T\sum_{n\leq T}\frac{|a_{f}(n)|^{2}}{n^{2\sigma}}\left(e^{-\frac{n}{T}}-1\right)^{2}+\sum_{n\leq T}\frac{|a_{f}(n)|^{2}}{n^{2\sigma-1}}\left(e^{-\frac{n}{T}}-1\right)^{2}\\
 & \ll\frac{1}{T}\,\sum_{n\leq T}\frac{|a_{f}(n)|^{2}}{n^{2\sigma-2}}+\frac{1}{T^{2}}\sum_{n\leq T}\frac{|a_{f}(n)|^{2}}{n^{2\sigma-3}}\\
 & \ll T^{k+\frac{3}{2}-2\sigma},
\end{align*}
while the third estimate, (\ref{third estimate outside line}), simply comes from the straightforward simplifications
\begin{align*}
\intop_{\frac{T}{2}}^{T}\left|\mathcal{J}_{3}\left(\sigma+it\right)\right|^{2}dt & =\intop_{\frac{T}{2}}^{T}\left|\sum_{n>T}\frac{a_{f}(n)}{n^{\sigma+it}}e^{-\frac{n}{T}}\right|^{2}dt=\sum_{n>T}\frac{|a_{f}(n)|^{2}}{n^{2\sigma}}e^{-\frac{2n}{T}}\left(\frac{T}{2}+O(n)\right)\\
 & \ll1.
\end{align*}
Using Stirling's formula and assuming that $\frac{k}{2}+\frac{1}{4}<\sigma<\frac{k}{2}+\frac{3}{4}$,
we can write the mean value estimate for $\mathcal{J}_{4}(\sigma+it)$
in the form
\begin{align*}
\intop_{\frac{T}{2}}^{T}\left|\mathcal{J}_{4}\left(\sigma+it\right)\right|^{2}dt & =\intop_{\frac{T}{2}}^{T}\left|(-1)^{\ell}\pi^{2s-k-\frac{1}{2}}\frac{\Gamma\left(k+\frac{1}{2}-\sigma-it\right)}{\Gamma(\sigma+it)}\,\sum_{n\leq T}\frac{a_{f}(n)}{n^{k+\frac{1}{2}-\sigma-it}}\right|^{2}dt\\
 & \ll T^{2k+1-4\sigma}\intop_{T/2}^{T}\left|\sum_{n\leq T}\frac{a_{f}(n)}{n^{k+\frac{1}{2}-\sigma-it}}\right|^{2}dt\ll T^{2k+1-4\sigma}\sum_{n\leq T}\frac{|a_{f}(n)|^{2}}{n^{2k+1-2\sigma}}\left(\frac{T}{2}+O(n)\right)\\
 & \ll T^{k+\frac{3}{2}-2\sigma},
\end{align*}
which proves (\ref{second estimate claim 2}). The mean value estimates
of $\mathcal{J}_{5}(s)$ and $\mathcal{J}_{6}(s)$ are analogous to
the estimates (\ref{J5(t)2}) and (\ref{J6 square claim}). We just
give the main steps of the proof of (\ref{J5 estimate outside strip}), as the proof of (\ref{J6 estimate outside strip line}) is completely
analogous. Following (\ref{J5(t)2}), (\ref{at most combining}) and
(\ref{estimate first guy}),
\begin{align*}
\intop_{T/2}^{T}\left|\mathcal{J}_{5}(\sigma+it)\right|^{2}dt & \ll\intop_{T/2}^{T}\,\intop_{-\eta-i\log^{2}|t|}^{-\eta+i\log^{2}|t|}\left|\frac{\Gamma\left(k+\frac{1}{2}-\sigma-it-w\right)\Gamma\left(\frac{w}{h}\right)}{\Gamma\left(\sigma+it+w\right)}\left(\pi^{2}T\right)^{w}\right|^{2}\left|\sum_{n>T}\frac{a_{f}(n)}{n^{k+\frac{1}{2}-s-w}}\right|^{2}\,|dw|\,dt\\
 & \ll T^{2k+1-4\sigma+2\eta}\,\intop_{-\eta-i\log^{2}(T)}^{-\eta+i\log^{2}(T)}\left|\Gamma(w)\right|^{2}\,\intop_{T/2}^{T}\left|\sum_{n>T}\frac{a_{f}(n)}{n^{k+\frac{1}{2}-s-w}}\right|^{2}dt\,|dw|\\
 & \ll T^{2k+1-4\sigma+2\eta}\,\sum_{n>T}\frac{|a_{f}(n)|^{2}}{n^{2k+1+2\eta-2\sigma}}\left(T+n\right)\\
 & =T^{2k+2-4\sigma+2\eta}\intop_{T}^{\infty}t^{2\sigma-2k-2\eta-1}\,d\left(\sum_{T<n\leq t}|a_{f}(n)|^{2}\right)\\
 & +T^{2k+1-4\sigma+2\eta}\,\intop_{T}^{\infty}t^{2\sigma-2k-2\eta}\,d\left(\sum_{T<n\leq t}|a_{f}(n)|^{2}\right).
\end{align*}
As before, we can perform and integration by parts and get
\begin{align}
\intop_{T}^{\infty}t^{2\sigma-2k-2\eta-1}\,d\left(\sum_{T<n\leq t}|a_{f}(n)|^{2}\right) & =\lim_{\tau\rightarrow\infty}\tau^{2\sigma-2k-2\eta-1}\sum_{T<n\leq\tau}|a_{f}(n)|^{2}\nonumber \\
 & +\left(2k+1+2\eta-2\sigma\right)\,\intop_{T}^{\infty}\sum_{T<n\leq t}|a_{f}(n)|^{2}\,t^{2\sigma-2k-2\eta-2}\,dt\nonumber \\
 & \ll T^{2\sigma-k-2\eta-\frac{1}{2}},\label{estimate first guy-1}
\end{align}
because, due to (\ref{mean discrete}), $\sum_{n\leq\tau}|a_{f}(n)|^{2}\ll\tau^{k+\frac{1}{2}}$. Using exactly the same procedure, one is able to obtain 
\begin{equation}
\intop_{T}^{\infty}t^{2\sigma-2k-2\eta}\,d\left(\sum_{T<n\leq t}|a_{f}(n)|^{2}\right)\ll T^{2\sigma-k-2\eta+\frac{1}{2}}.\label{estimate second guy 2}
\end{equation}
Combining (\ref{estimate first guy-1}) and (\ref{estimate second guy 2})
yields the formula 
\[
\intop_{T/2}^{T}\left|\mathcal{J}_{5}(\sigma+it)\right|^{2}dt\ll T^{k+\frac{3}{2}-2\sigma}
\]
and a similar proof can be written for the sixth estimate (\ref{J6 estimate outside strip line}). This
completes the proof of Claim \ref{claim 2.2}.
\end{proof}

\begin{claim}\label{claim 2.3}
Let $\frac{T}{2}<t\leq T$ and $\mathcal{J}_{\ell}(s)$, $\ell=1,2,...,6$
be the quantities defined by (\ref{More general gormula!}). Then the following bounds are valid
\begin{equation}
\mathcal{J}_{1}\left(\sigma+it\right)\ll T^{\frac{k}{2}+\frac{3}{4}-\sigma},\,\,\,\,0<\sigma<\frac{k}{2}+\frac{1}{4},\label{J1 strip estimate}
\end{equation}
\begin{equation}
\mathcal{J}_{2}\left(\sigma+it\right)\ll T^{\frac{k}{2}+\frac{3}{4}-\sigma},\,\,\,\,0<\sigma<\frac{k}{2}+\frac{3}{4},\label{J2 strip estimate}
\end{equation}
\begin{equation}
\mathcal{J}_{3}\left(\sigma+it\right)\ll1,\,\,\,\,0<\sigma<\frac{k}{2}+\frac{3}{4},\label{J3 strip estimate}
\end{equation}
\begin{equation}
\mathcal{J}_{4}(\sigma+it)\ll T^{\frac{k}{2}+\frac{3}{4}-\sigma},\,\,\,\,\frac{k}{2}+\frac{1}{4}<\sigma<\frac{k}{2}+\frac{3}{4},\label{J4 strip estimate}
\end{equation}
\begin{equation}
\mathcal{J}_{5}\left(\sigma+it\right)\ll T^{\frac{k}{2}+\frac{3}{4}-\sigma},\,\,\,\,0<\sigma<\frac{k}{2}+\frac{3}{4},\label{J5 estimate strip}
\end{equation}
\begin{equation}
\mathcal{J}_{6}\left(\sigma+it\right)\ll T^{\frac{k}{2}+\frac{3}{4}-\sigma},\,\,\,\,0<\sigma<\frac{k}{2}+\frac{3}{4}.\label{J6 strip estimate}
\end{equation}
\end{claim}

\begin{proof}
By the Cauchy-Schwarz inequality and (\ref{mean discrete}), we see
that
\begin{equation}
|\mathcal{J}_{1}(\sigma+it)|\leq\sum_{n\leq T}\frac{|a_{f}(n)|}{n^{\sigma}}\leq T^{\frac{1}{2}}\left(\sum_{n\leq T}\frac{|a_{f}(n)|^{2}}{n^{2\sigma}}\right)^{\frac{1}{2}}\ll T^{\frac{k}{2}+\frac{3}{4}-\sigma}.\label{intermediate 1}
\end{equation}
Using a similar reasoning, we can simply bound $|\mathcal{J}_{2}(\sigma+it)|$
\begin{equation}
|\mathcal{J}_{2}(\sigma+it)|\leq\frac{1}{T}\sum_{n\leq T}\frac{|a_{f}(n)|}{n^{\sigma-1}}\ll T^{-\frac{1}{2}}\left(\sum_{n\leq T}\frac{|a_{f}(n)|^{2}}{n^{2\sigma-2}}\right)^{\frac{1}{2}}\ll T^{\frac{k}{2}+\frac{3}{4}-\sigma},\label{intermediate 2}
\end{equation}
while $|\mathcal{J}_{3}\left(\sigma+it\right)|\ll1$. To deal with
$\mathcal{J}_{4}(\sigma+it)$, we appeal to Stirling's formula and
use the same inequalities that were used to derive (\ref{intermediate 1})
and (\ref{intermediate 2}). We obtain
\begin{align*}
|\mathcal{J}_{4}(\sigma+it)| & =\left|(-1)^{\ell}\pi^{2s-k-\frac{1}{2}}\frac{\Gamma\left(k+\frac{1}{2}-\sigma-it\right)}{\Gamma(\sigma+it)}\,\sum_{n\leq T}\frac{a_{f}(n)}{n^{k+\frac{1}{2}-\sigma-it}}\right|\\
 & \ll T^{k+1-2\sigma}\,\left(\sum_{n\leq T}\frac{|a_{f}(n)|^{2}}{n^{2k+1-2\sigma}}\right)^{\frac{1}{2}}\ll T^{\frac{k}{2}+\frac{3}{4}-\sigma}.
\end{align*}
Finally, $\mathcal{J}_{5}(\sigma+it)$ can be estimated by the integral
\[
|\mathcal{J}_{5}(\sigma+it)|\ll T^{k+\frac{1}{2}-2\sigma+\eta}\,\intop_{-\eta-i\log^{2}|t|}^{-\eta+i\log^{2}|t|}\left|\Gamma\left(\frac{w}{h}\right)\right|\,\sum_{n>T}\frac{|a_{f}(n)|}{n^{k+\frac{1}{2}+\eta-\sigma}}\,|dw|.
\]
In order to bound the series on the integrand, we use once more summation
by parts to get
\begin{align}
\sum_{n>T}\frac{|a_{f}(n)|}{n^{k+\frac{1}{2}+\eta-\sigma}} & =\intop_{T}^{\infty}t^{-k-\frac{1}{2}-\eta+\sigma}\,d\left(\sum_{T<n\leq t}|a_{f}(n)|\right)\nonumber \\
=\lim_{\tau\rightarrow\infty}\tau^{-k-\frac{1}{2}-\eta+\sigma}\sum_{T<n\leq\tau}|a_{f}(n)| & +\left(k+\frac{1}{2}+\eta-\sigma\right)\,\intop_{T}^{\infty}t^{-k-\frac{3}{2}-\eta+\sigma}\sum_{T<n\leq t}|a_{f}(n)|\,dt.\label{after an integration by parts....}
\end{align}
By the Cauchy-Schwarz inequality,
\[
\tau^{-k-\frac{1}{2}-\eta+\sigma}\sum_{T<n\leq\tau}|a_{f}(n)|\ll\tau^{-k-\eta+\sigma}\,\left(\sum_{n\leq\tau}|a_{f}(n)|^{2}\right)^{\frac{1}{2}}=\tau^{-\frac{k}{2}+\frac{1}{4}-\eta+\sigma},
\]
which tends to zero as $\tau\rightarrow\infty$ because $\sigma<\frac{k}{2}-\frac{1}{4}+\eta$
by construction of $\eta$. For the same reasons, the integral in
(\ref{after an integration by parts....}) can be estimated in the
form
\[
\intop_{T}^{\infty}t^{-k-\frac{3}{2}-\eta+\sigma}\sum_{T<n\leq t}|a_{f}(n)|\,dt\ll T^{-\frac{k}{2}+\frac{1}{4}-\eta+\sigma},
\]
which finally proves (\ref{J5 estimate strip}). The proof of the
corresponding estimate for $\mathcal{J}_{6}(\sigma+it)$, (\ref{J6 strip estimate}),
is totally analogous.
\end{proof}

Returning to (\ref{after approximate functional in the proof}) and
using the results of Claim \ref{claim 2.1}, we have that
\begin{equation}
\intop_{T/2}^{T}\left|L\left(\frac{k}{2}+\frac{1}{4}+it,f\right)\right|^{2}dt=r_{f}T\log(T)+O\left(T\right)+O\left(\sum_{m\neq\ell}\intop_{\frac{T}{2}}^{T}J_{m}(t)\,\overline{J_{\ell}(t)}\,dt\right).\label{almost at the end mean}
\end{equation}
Hence, to complete our proof, we just need to estimate the following
integrals
\begin{equation}
\intop_{\frac{T}{2}}^{T}J_{m}(t)\,\overline{J_{\ell}(t)}\,dt=\frac{1}{i}\,\intop_{\frac{k}{2}+\frac{1}{4}+i\frac{T}{2}}^{\frac{k}{2}+\frac{1}{4}+iT}\mathcal{J}_{m}\left(z\right)\,\mathcal{J}_{\ell}\left(k+\frac{1}{2}-z\right)\,dz,\,\,\,\,\ell\neq m.\label{start general interal as contours}
\end{equation}
Assume first that $m=1$: shifting the line of integration from $\text{Re}(z)=\frac{k}{2}+\frac{1}{4}$
to $\text{Re}(z)=\beta$, $0<\beta<\frac{k}{2}+\frac{1}{4}$, and
employing Cauchy's theorem, we find that
\begin{equation}
\intop_{\frac{T}{2}}^{T}J_{1}(t)\,\overline{J_{\ell}(t)}\,dt=\frac{1}{i}\,\intop_{\beta+i\frac{T}{2}}^{\beta+iT}\mathcal{J}_{1}\left(z\right)\,\mathcal{J}_{\ell}\left(k+\frac{1}{2}-z\right)\,dz+O\left(\intop_{\frac{k}{2}+\frac{1}{4}}^{\beta}\mathcal{J}_{1}\left(\sigma+iT\right)\,\mathcal{J}_{\ell}\left(k+\frac{1}{2}-\sigma-iT\right)\,d\sigma\right).\label{shift line}
\end{equation}
By Claim \ref{claim 2.3}, we know that the integral along the horizontal segment
$\beta\leq\sigma\leq\frac{k}{2}+\frac{1}{4}$ is bounded as $\ll T$.
Next, by Claim \ref{claim 2.2} and the Cauchy-Schwarz inequality, one sees that
\begin{align*}
\left|\intop_{\alpha+i\frac{T}{2}}^{\alpha+iT}\mathcal{J}_{1}\left(z\right)\,\mathcal{J}_{\ell}\left(k+\frac{1}{2}-z\right)\,dz\right| & \ll\intop_{\frac{T}{2}}^{T}\left|\mathcal{J}_{1}\left(\alpha+it\right)\,\mathcal{J}_{\ell}\left(k+\frac{1}{2}-\alpha-it\right)\right|\,dt\\
 & \ll T,
\end{align*}
which establishes the estimate
\begin{equation}
\intop_{\frac{T}{2}}^{T}J_{1}(t)\,\overline{J_{\ell}(t)}\,dt\ll T.\label{J1 estimate cross}
\end{equation}
Note that we had to move the line of integration to the left of the
line $\text{Re}(z)=\frac{k}{2}+\frac{1}{4}$ because (\ref{1st estimate claim 2})
and (\ref{J1 strip estimate}) are only valid in the range $0<\sigma<\frac{k}{2}+\frac{1}{4}$.
Note that the case $\ell=4$ is not problematic because $\frac{k}{2}+\frac{1}{4}<k+\frac{1}{2}-\beta<\frac{k}{2}+\frac{3}{4}$
and so the application of the estimates (\ref{second estimate claim 2})
and (\ref{J4 strip estimate}) is valid. If, on the other hand, we start the estimation of (\ref{start general interal as contours}) when $m=4$,
then we have to shift the line of integration in (\ref{start general interal as contours})
from the critical line $\text{Re}(z)=\frac{k}{2}+\frac{1}{4}$ to $\text{Re}(z)=\beta^{\prime}$, $\frac{k}{2}+\frac{1}{4}<\text{Re}(z)=\beta^{\prime}<\frac{k}{2}+\frac{3}{4}$, 
so that the application of the estimates (\ref{second estimate claim 2}) and (\ref{J4 strip estimate}) is fully justified in this case. The same kind of procedure gives
\begin{equation}
\intop_{\frac{T}{2}}^{T}J_{4}(t)\,\overline{J_{\ell}(t)}\,dt\ll T.\label{J4 estimate cross}
\end{equation}
All the remaining cross terms (i.e., those with $m\neq 1,4$) can be bounded in a similar fashion and
their estimates contribute with a $O(T)$ term. A combination of (\ref{almost at the end mean}),
(\ref{J1 estimate cross}), (\ref{J4 estimate cross}) and all the
remaining cross terms yields
\begin{equation}
\intop_{T/2}^{T}\left|L\left(\frac{k}{2}+\frac{1}{4}+it,f\right)\right|^{2}dt=r_{f}T\log(T)+O\left(T\right).\label{Final dyadic}
\end{equation}
Our lemma \ref{ramachandra mean} follows now from (\ref{Final dyadic}) after one replaces
$T$ by $T/2$, $T/4,$$T/8$,..., and finally adds all the results. 

\end{proof}

\section{Proof of Landau's formula}

In \cite{Landau_I}, Landau proved the explicit formula
\begin{equation}
\sum_{0<\gamma\leq T}x^{\rho}=-\Lambda(x)\,\frac{T}{2\pi}+O\left(\log(T)\right),\,\,\,\,\,\,T\rightarrow\infty,\label{landau formula x>1}
\end{equation}
where $x>1$ is fixed and the sum runs over the nontrivial zeros of
$\zeta(s)$, $\rho=\beta+i\gamma$. Here, $\Lambda(x)$ denotes von-
Mangoldt's function. A formula valid in the range $0<x<1$ follows
after we replace $x$ by $1/x$, multiplying the resulting formula
by $x$ and observing that $\zeta(\rho)=0$ iff $\zeta(1-\rho)=0$.
Thus, we can rewrite (\ref{landau formula x>1}) taking into account
the combination of both ranges for $x$ and see that
\begin{equation}
\sum_{0<\gamma\leq T}x^{\rho}=-\left(\Lambda(x)+x\Lambda\left(\frac{1}{x}\right)\right)\,\frac{T}{2\pi}+O\left(\log(T)\right),\,\,\,\,\,\,T\rightarrow\infty\label{both ranges combined}
\end{equation}
holds for any $x\in\mathbb{R}^{+}\setminus\{1\}$. Steuding proved
a formula similar to (\ref{both ranges combined}) for $a-$points
{[}\cite{steuding_100}, Thm. 6{]} and Baluyot and Gonek \cite{gonek_baluyot}
proved a uniform version of Steuding's formula.

In this section, we are going to prove an analogue of (\ref{both ranges combined})
for our class of $L-$functions. 

We first make a couple of observations. Let $c$ be the smallest integer such that $a_{f}(c)\neq0$: since
the Dirichlet series defining $L(s,f)$, (\ref{Dirichlet L series cusp form at intro first}), converges in the half-plane
$\text{Re}(s)>\frac{k}{2}+\frac{3}{4}$, there exists some large constant,
say $B$, such that, for any $\text{Re}(s)\geq B$,
\begin{equation}
\sum_{n=c+1}^{\infty}\frac{|a_{f}(n)|}{n^{\text{Re}(s)}}\leq\frac{1}{2}|a_{f}(c)|c^{-\text{Re}(s)},\label{cndition no zeros}
\end{equation}
which shows that $L(s,f)$ does not have zeros in the half-plane $\text{Re}(s)\geq B$
because condition (\ref{cndition no zeros}) yields
\begin{equation}
|L(s,f)|\geq\frac{1}{2}|a_{f}(c)|c^{-\alpha},\,\,\,\,\,\text{Re}(s)\geq\alpha\geq B.\label{zero free region}
\end{equation}

Consider now the modified Dirichlet series
\[
\tilde{L}(s,f):=\frac{c^{s}}{a_{f}(c)}\,L(s,f).
\]
Clearly, $\tilde{L}(s,f)$ converges absolutely for $\text{Re}(s)>\frac{k}{2}+\frac{3}{4}$,
just like $L(s,f)$ does. Moreover, the zeros of $\tilde{L}(s,f)$
and $L(s,f)$ are the same and, of course, the zero free region obtained
from (\ref{zero free region}) is the same. Moreover, $\tilde{L}(s,f)$
satisfies the functional equation (cf. (\ref{functional equation iwht l})
above)
\begin{equation}
\tilde{L}(s,f):=\frac{c^{s}}{a_{f}(c)}\,L(s,f)=(-1)^{\ell}\,(\pi c)^{2s-k-\frac{1}{2}}\frac{\Gamma\left(k+\frac{1}{2}-s\right)}{\Gamma(s)}\,\tilde{L}\left(k+\frac{1}{2}-s,f\right).\label{functional equation for tilde L}
\end{equation}

Let $R$ be the positively oriented rectangular contour with vertices
$B\pm iT$ and $k+\frac{1}{2}-b\pm iT$, for some $b>B$. An application
of the residue theorem gives
\begin{equation}
\sum_{|\gamma|<T}x^{\rho}=\frac{1}{2\pi i}\intop_{R}\frac{\tilde{L}^{\prime}\left(s,f\right)}{\tilde{L}\left(s,f\right)}x^{s}ds,\label{Starting point Landu}
\end{equation}
where $\rho=\beta+i\gamma$ runs over the zeros of $L(s,f)$ inside
$R$. All but finitely many of the zeros of $\tilde{L}(s,f)$ inside
$R$ are nontrivial and the only trivial ones are the nonpositive
integers running from $0$ to $k+\frac{1}{2}-b$ [\cite{berndt_zeros_(i)}, Theorem 1]. Next, note that,
for $\text{Re}(s)\geq\alpha\geq B$, the condition (\ref{zero free region})
establishes that $|\tilde{L}(s,f)-1|\leq\frac{1}{2}$. Hence, since
$\tilde{L}(s,f)-1$ converges absolutely for $\sigma\geq B$, we can
write the $\ell^{\text{th}}$ power of $\tilde{L}(s,f)-1$ in the
following form
\[
\left(\tilde{L}(s,f)-1\right)^{\ell}=\frac{c^{\ell s}}{a_{f}(c)^{\ell}}\sum_{n_{1},...,n_{\ell}\geq c+1}\frac{a_{f}(n_{1})a_{f}(n_{2})...a_{f}(n_{\ell})}{\left(n_{1}\cdot n_{2}\cdot...\cdot n_{\ell}\right)^{s}}.
\]
Therefore, when $\sigma=\text{Re}(s)\geq B$, the reciprocal of $\tilde{L}(s,f)$
admits the representation
\[
\frac{1}{\tilde{L}(s,f)}=\sum_{\ell=0}^{\infty}(-1)^{\ell}\left(\tilde{L}(s,f)-1\right)^{\ell}=\sum_{\ell=0}^{\infty}\frac{(-1)^{\ell}c^{\ell s}}{a_{f}(c)^{\ell}}\sum_{n_{1},...,n_{\ell}\geq c+1}\frac{a_{f}(n_{1})a_{f}(n_{2})...a_{f}(n_{\ell})}{\left(n_{1}\cdot n_{2}\cdot...\cdot n_{\ell}\right)^{s}}.
\]
Note that, as a multiple Dirichlet series, the previous expression
converges absolutely when $\text{Re}(s)\geq B$: indeed, by (\ref{zero free region}),
\[
\sum_{\ell=0}^{\infty}\frac{c^{\ell\sigma}}{|a_{f}(c)|^{\ell}}\sum_{n_{1}\geq c+1}\frac{|a_{f}(n_{1})|}{n_{1}^{\sigma}}\cdot...\cdot\sum_{n_{\ell}\geq c+1}\frac{|a_{f}(n_{\ell})|}{n_{\ell}^{\sigma}}\leq\sum_{\ell=0}^{\infty}\frac{c^{\ell\sigma}}{|a_{f}(c)|^{\ell}}\,\left(\frac{1}{2}|a_{f}(c)|c^{-\sigma}\right)^{\ell}=2,
\]
which shows that we can write $1/\tilde{L}(s,f)$ as a single Dirichlet
series, absolutely convergent in the region $\text{Re}(s)\geq B$,
and whose representation assumes the form
\begin{equation}
\frac{1}{\tilde{L}(s,f)}=\sum_{n=1}^{\infty}\frac{d_{f}(n)}{\mu_{n}^{s}},\label{Reciprocal L function cusp}
\end{equation}
where $1=\mu_{1}<\mu_{2}<...$ is an ordered sequence of the set of
all possible products of the natural numbers greater than $c$ (cf. \cite{Landau_II, selberg_zeta_apoints}). Multiplying
(\ref{Reciprocal L function cusp}) by the Dirichlet series representing
$\tilde{L}^{\prime}(s,f)$ (which is also absolutely convergent when
$\text{Re}(s)\geq B$), we are able to get 
\begin{equation}
\frac{\tilde{L}^{\prime}\left(s,f\right)}{\tilde{L}\left(s,f\right)}:=\sum_{n=1}^{\infty}\frac{b_{f}(n)}{\lambda_{n}^{s}},\,\,\,\,\,\,\text{Re}(s)\geq B\label{expansion as Dirichlet for B}
\end{equation}
for some sequence $c=\lambda_{1}<\lambda_{2}<...$.

Before moving to the main result of this section, we need to state the following lemma, which can be found in [\cite{titchmarsh_zetafunction}, p. 56]. 
\begin{lemma}\label{Lemma alpha}
If $f(s)$ is an analytic function such that $f(s_{0})\neq0$ and,
for some $M>1$ and $r>0$, 
\begin{equation}
\left|\frac{f(s)}{f(s_{0})}\right|<e^{M},\,\,\,\,\,\,|s-s_{0}|\leq r.\label{condition in lemma alpha Titch}
\end{equation}
Then, for any $s$ satisfying $|s-s_{0}|\leq\frac{r}{4}$, the inequality
holds
\begin{equation}
\left|\frac{f^{\prime}(s)}{f(s)}-\sum_{|\rho-s_{0}|\leq\frac{r}{2}}\,\frac{1}{s-\rho}\right|<C\,\frac{M}{r},\label{inequality lemma jensen type}
\end{equation}
where $C>0$ and $\rho$ runs through the zeros of $f(s)$.
\end{lemma}

With the observations above, we are ready to prove the following proposition.

\begin{proposition}\label{landau proposition}
Let $x\in\mathbb{R}^{+}\setminus\{1\}$ and $d\in\mathbb{N}$. Then,
as $T\rightarrow\infty$, the following formula holds
\begin{equation}
  \sum_{|\gamma|<T}x^{\rho}=\frac{b_{f}(d)}{\pi}\left(\delta_{x,\lambda_{d}}-x^{k+\frac{1}{2}}\delta_{x,\lambda_{d}^{-1}}\right)T+O\left(\log(T)\right), \label{formula Landau explicit}  
\end{equation}
where $\rho=\beta+i\gamma$, $|\gamma|<T$, runs over the zeros of
$L(s,f)$ on the rectangle $|\text{Im}(z)|<T$. Moreover, $b_{f}(n)$
is the arithmetical function defined by (\ref{expansion as Dirichlet for B}).
\end{proposition}

\begin{proof}
By (\ref{Starting point Landu}), we have that
\begin{align}
\sum_{|\gamma|<T}x^{\rho} & =\frac{1}{2\pi i}\left\{ \intop_{B-iT}^{B+iT}+\intop_{B+iT}^{k+\frac{1}{2}-b+iT}+\intop_{k+\frac{1}{2}-b+iT}^{k+\frac{1}{2}-b-iT}+\intop_{k+\frac{1}{2}-b-iT}^{B-iT}\right\} \,\frac{\tilde{L}^{\prime}\left(s,f\right)}{\tilde{L}\left(s,f\right)}\,x^{s}ds\nonumber \\
 & :=\sum_{j=1}^{4}I_{j}(x).\label{to start with}
\end{align}

In the next claims we shall bound the integrals $I_{1}(x)$, $I_{2}(x)$,
$I_{3}(x)$ and $I_{4}(x)$.

\begin{claim}
For $x\in\mathbb{R}^{+}\setminus\{1\}$, let 
\[
I_{1}(x)=\frac{1}{2\pi i}\,\intop_{B-iT}^{B+iT}\frac{\tilde{L}^{\prime}\left(s,f\right)}{\tilde{L}\left(s,f\right)}\,x^{s}ds.
\]
Then $I_{1}(x)$ satisfies the following estimate 
\begin{equation}
I_{1}(x)=\frac{b_{f}(d)}{\pi}\delta_{x,\lambda_{d}}\,T+O(1),\label{I1(x) estimate}
\end{equation}
where $\lambda_{d}$ is the $d^{\text{th}}$ element of the sequence
$\left(\lambda_{n}\right)_{n\in\mathbb{N}}$ given in the Dirichlet
series (\ref{expansion as Dirichlet for B}).
\end{claim}

\begin{proof}
Starting with the first integral, we see that, since (\ref{expansion as Dirichlet for B})
is valid for $\text{Re}(s)\geq B$,
\begin{align*}
I_{1}(x) & =\frac{1}{2\pi i}\,\intop_{B-iT}^{B+iT}\frac{\tilde{L}^{\prime}\left(s,f\right)}{\tilde{L}\left(s,f\right)}\,x^{s}ds=\frac{1}{2\pi}\sum_{n=1}^{\infty}b_{f}(n)\,\left(\frac{x}{\lambda_{n}}\right)^{B}\intop_{-T}^{T}\left(\frac{x}{\lambda_{n}}\right)^{iT}dt\\
 & =\frac{1}{2\pi}\sum_{n=1}^{\infty}b_{f}(n)\,\left(\frac{x}{\lambda_{n}}\right)^{B}\left\{ 2\delta_{x,\lambda_{n}}T+O(1)\right\} \\
 & =\frac{1}{\pi}\delta_{x,\lambda_{d}}\,b_{f}(d)T+O(1),
\end{align*}
which completes the proof of (\ref{I1(x) estimate}).
\end{proof}

\begin{claim}
For $x\in\mathbb{R}^{+}\setminus\{1\}$, let $I_{2}(x)$ and $I_{4}(x)$
denote the contour integrals
\[
I_{2}(x)=\frac{1}{2\pi i}\,\intop_{B+iT}^{k+\frac{1}{2}-b+iT}\frac{\tilde{L}^{\prime}\left(s,f\right)}{\tilde{L}\left(s,f\right)}\,x^{s}ds
\]
and
\[
I_{4}(x)=\frac{1}{2\pi i}\,\intop_{k+\frac{1}{2}-b-iT}^{B-iT}\frac{\tilde{L}^{\prime}\left(s,f\right)}{\tilde{L}\left(s,f\right)}\,x^{s}ds.
\]

Then $I_{2}(x)$ and $I_{4}(x)$ satisfy the bounds
\begin{equation}
I_{2}(x),I_{4}(x)\ll x^{B}\log(T).\label{Estimates I2 and I4}
\end{equation}
\end{claim}

\begin{proof}

We start our proof by evaluating first the integral
\begin{equation}
I_{2}(x)=\frac{1}{2\pi i}\,\intop_{B+iT}^{k+\frac{1}{2}-b+iT}\frac{\tilde{L}^{\prime}\left(s,f\right)}{\tilde{L}\left(s,f\right)}\,x^{s}ds=-\frac{x^{iT}}{2\pi i}\intop_{k+\frac{1}{2}-b}^{B}\frac{\tilde{L}^{\prime}\left(\sigma+iT,f\right)}{\tilde{L}\left(\sigma+iT,f\right)}\,x^{\sigma}d\sigma,\label{Definition of integral I2}
\end{equation}
by appealing to Lemma \ref{Lemma alpha}. Indeed, applying this lemma with $f(s)=\tilde{L}(s,f)$
and $s_{0}=B+iT$, we see that $\tilde{L}(s_{0},f)\neq0$ by (\ref{zero free region}).
Moreover, if $r=4\left(B+b-k-\frac{1}{2}\right)$, then $|s-s_{0}|\leq r$
for $s=\sigma+iT$, $k+\frac{1}{2}-b\leq\sigma\leq B$. Thus, by convex
estimates, we find that 
\begin{equation}
\left|\frac{\tilde{L}^{\prime}\left(\sigma+iT,f\right)}{\tilde{L}\left(B+iT,f\right)}\right|\ll T^{A(\sigma)}\ll e^{A\,\log(T)},\label{Condition}
\end{equation}
for some suitably large constant $A$. Hence, (\ref{condition in lemma alpha Titch})
holds with $M=A\,\log(T)$. An application of (\ref{inequality lemma jensen type}) now gives
\[
\frac{\tilde{L}^{\prime}\left(s,f\right)}{\tilde{L}\left(s,f\right)}=\sum_{|\rho-B-iT|\leq\frac{r}{2}}\frac{1}{s-\rho}+O\left(\log(T)\right),
\]
whenever $|s-B-iT|\leq B+b-k-\frac{1}{2}$. Returning to $I_{2}(x)$, we find that
\begin{align*}
|I_{2}(x)| & \ll x^{B}\intop_{k+\frac{1}{2}-b}^{B}\left|\frac{\tilde{L}^{\prime}\left(\sigma+iT,f\right)}{\tilde{L}\left(\sigma+iT,f\right)}\right|\,d\sigma\leq x^{B}\intop_{k+\frac{1}{2}-b}^{B}\sum_{|\rho-B-iT|\leq\frac{r}{2}}\frac{d\sigma}{|\sigma+iT-\rho|}+O\left(x^{B}\log(T)\right)\\
 & =x^{B}\sum_{|\rho-B-iT|\leq\frac{r}{2}}\,\intop_{k+\frac{1}{2}-b}^{B}\,\frac{d\sigma}{\sqrt{\left(\sigma-\beta\right)^{2}+\left(T-\gamma\right)^{2}}}+O\left(x^{B}\log(T)\right)\\
 & =\frac{1}{2}x^{B}\sum_{|\rho-B-iT|\leq\frac{r}{2}}\,\left[\log\left(\frac{\sqrt{(\sigma-\beta)^{2}+(T-\gamma)^{2}}+\sigma-\beta}{\sqrt{(\sigma-\beta)^{2}+(T-\gamma)^{2}}+\beta-\sigma}\right)\right]_{k+\frac{1}{2}-b}^{B}+O\left(x^{B}\log(T)\right)\\
 & =\frac{1}{2}x^{B}\sum_{|\rho-B-iT|\leq\frac{r}{2}}\log\left(\frac{\sqrt{(B-\beta)^{2}+(T-\gamma)^{2}}+B-\beta}{\sqrt{(k+\frac{1}{2}-b-\beta)^{2}+(T-\gamma)^{2}}+k+\frac{1}{2}-b-\beta}\right)\\
&\,\,\,\,\,\,\,\,\,\,\,\,\,\,\,\,\,\,\,\,\,\,\,\,\,\,\,\,\,\,\,\,\,\,\,\,\,\,\,\,\,\,\,+\log\left(\frac{\sqrt{(k+\frac{1}{2}-b-\beta)^{2}+(T-\gamma)^{2}}+\beta+b-k-\frac{1}{2}}{\sqrt{(B-\beta)^{2}+(T-\gamma)^{2}}+\beta-B}\right)+O(x^{B}\log(T)).
\end{align*}
The summand in the above sum is clearly bounded, so that our estimate
gives
\[
|I_{2}(x)|\ll x^{B}\sum_{|\rho-B-iT|\leq\frac{r}{2}}1+O\left(\log(T)\right)=x^{B}N_{C}+O(x^{B}\log(T)),
\]
where $N_{C}$ denotes the number of zeros of $L(s,f)$ satisfying
the condition $|\rho-B-iT|\leq\frac{r}{2}$. Using the analogue of
the Riemann Von-Mangoldt formula for $L(s,f)$ [\cite{berndt_zeros_(i)}, Theorem 3], one can easily show
that $N_{C}\ll\log(T)$, so that
\[
|I_{2}(x)|\ll x^{B}\log(T),
\]
which completes the proof of the first part of (\ref{Estimates I2 and I4}). The proof
of the estimate for $I_{4}(x)$ is analogous.
\end{proof}

\begin{claim}
For $x\in\mathbb{R}^{+}\setminus\{1\}$, let $I_{3}(x)$ be the integral
given by
\begin{equation}
I_{3}(x)=\frac{1}{2\pi i}\intop_{k+\frac{1}{2}-b-iT}^{k+\frac{1}{2}-b+iT}\frac{\tilde{L}^{\prime}(s,f)}{\tilde{L}(s,f)}x^{s}ds.\label{I3 (x) def}
\end{equation}

Then the following asymptotic formula takes place
\begin{equation}
I_{3}(x)=-\frac{Tx^{k+\frac{1}{2}}}{\pi}\,b_{f}\left(d\right)\delta_{x,\lambda_{d}^{-1}}+O\left(\log(T)\right).\label{I3(x) asympto}
\end{equation}
\end{claim}

%\begin{align*}
%\log\tilde{L}(s,f) & =\left(2s-k-\frac{1}{2}\right)\log(\pi %c)+\log\Gamma\left(k+\frac{1}{2}-s\right)-%\log\Gamma(s)+\log\tilde{L}\left(k+\frac{1}{2}-s,f\right)\\
% & +i\pi\,\delta_{\ell,1}.
%\end{align*}
\begin{proof}
Once more, let us recall that $f|W_{4}(z)=(-1)^{\ell}f(z)$. From the functional equation satisfied by $\tilde{L}(s,f)$, (\ref{functional equation for tilde L}), 
\[
\frac{\tilde{L}^{\prime}(s,f)}{\tilde{L}(s,f)}=2\log(\pi c)-\frac{\Gamma^{\prime}\left(k+\frac{1}{2}-s\right)}{\Gamma\left(k+\frac{1}{2}-s\right)}-\frac{\Gamma^{\prime}(s)}{\Gamma(s)}-\frac{\tilde{L}^{\prime}\left(k+\frac{1}{2}-s,f\right)}{\tilde{L}\left(k+\frac{1}{2}-s,f\right)},
\]
which shows that the integral along the vertical segment $z\in\left[k+\frac{1}{2}-b-iT,k+\frac{1}{2}-b+iT\right]$
can be written in the following form
\begin{align*}
I_{3}(x) & =\frac{1}{2\pi i}\intop_{k+\frac{1}{2}-b+iT}^{k+\frac{1}{2}-b-iT}\frac{\tilde{L}^{\prime}(s,f)}{\tilde{L}(s,f)}x^{s}ds\\
=-\frac{1}{2\pi i}\intop_{k+\frac{1}{2}-b-iT}^{k+\frac{1}{2}-b+iT} & \left\{ 2\log(\pi c)-\frac{\Gamma^{\prime}\left(k+\frac{1}{2}-s\right)}{\Gamma\left(k+\frac{1}{2}-s\right)}-\frac{\Gamma^{\prime}(s)}{\Gamma(s)}-\frac{\tilde{L}^{\prime}\left(k+\frac{1}{2}-s,f\right)}{\tilde{L}\left(k+\frac{1}{2}-s,f\right)}\right\} \,x^{s}\,ds.
\end{align*}

Let us evaluate the integral with respect to the logarithmic derivative
of the funtion $\tilde{L}(z,f)$: by (\ref{zero free region}), the Dirichlet series for
$\tilde{L}^{\prime}(s,f)/\tilde{L}(s,f)$ converges absolutely when
$\text{Re}(s)=b>B$. Therefore, 
\begin{align}
\frac{1}{2\pi i}\,\intop_{k+\frac{1}{2}-b-iT}^{k+\frac{1}{2}-b+iT}\frac{\tilde{L}^{\prime}\left(k+\frac{1}{2}-s,f\right)}{\tilde{L}\left(k+\frac{1}{2}-s,f\right)}\,x^{s}ds & =x^{k+\frac{1}{2}}\sum_{n=1}^{\infty}\frac{b_{f}(n)}{2\pi i}\,\intop_{b-iT}^{b+iT}\left(x\lambda_{n}\right)^{-s}ds\nonumber \\
=\frac{x^{k-b+\frac{1}{2}}}{2\pi}\sum_{n=1}^{\infty}\frac{b_{f}(n)}{\lambda_{n}^{b}}\intop_{-T}^{T}\left(x\lambda_{n}\right)^{-it}dt & =\frac{x^{k-b+\frac{1}{2}}}{2\pi}\sum_{n=1}^{\infty}\frac{b_{f}(n)}{\lambda_{n}^{b}}\left\{ 2T\delta_{x,\lambda_{n}^{-1}}+O(1)\right\} \nonumber \\
 & =\frac{Tx^{k+\frac{1}{2}}}{\pi}\,b_{f}\left(d\right)\delta_{x,\lambda_{d}^{-1}}+O(1).\label{evaluation I3a}
\end{align}

Next, an immediate application of Stirling's formula, 
\[
\frac{\Gamma^{\prime}(s)}{\Gamma(s)}=\log(s)-\frac{1}{2s}+O\left(\frac{1}{|s|^{2}}\right),
\]
gives
\begin{equation}
\frac{1}{2\pi i}\intop_{k+\frac{1}{2}-b-iT}^{k+\frac{1}{2}-b+iT}\left\{ 2\log(\pi c)-\frac{\Gamma^{\prime}\left(k+\frac{1}{2}-s\right)}{\Gamma\left(k+\frac{1}{2}-s\right)}-\frac{\Gamma^{\prime}(s)}{\Gamma(s)}\right\} \,x^{s}\,ds=O\left(\log(T)\right).\label{Evaluation I3b}
\end{equation}
Combining (\ref{evaluation I3a}) and (\ref{Evaluation I3b}), one
is able to obtain (\ref{I3(x) asympto}).

\end{proof}

Combining the previous three claims and returning to the integral
representation (\ref{to start with}), we obtain an analogue of (\ref{landau formula x>1}),
\begin{align*}
\sum_{|\gamma|<T}x^{\rho} & =\frac{1}{2\pi i}\left\{ \intop_{B-iT}^{B+iT}+\intop_{B+iT}^{k+\frac{1}{2}-b+iT}+\intop_{k+\frac{1}{2}-b+iT}^{k+\frac{1}{2}-b-iT}+\intop_{k+\frac{1}{2}-b-iT}^{B-iT}\right\} \,\frac{\tilde{L}^{\prime}\left(s,f\right)}{\tilde{L}\left(s,f\right)}\,x^{s}ds\\
 & =\frac{b_{f}(d)}{\pi}\left(\delta_{x,\lambda_{d}}-x^{k+\frac{1}{2}}\delta_{x,\lambda_{d}^{-1}}\right)T+O\left(\log(T)\right),
\end{align*}
which is precisely (\ref{formula Landau explicit}).

\end{proof}

\section{Proof of Theorem \ref{theorem uniform distribution}}

By Weyl's criterion, we want to prove that, for every $m\in\mathbb{Z}\setminus\{0\}$,
\begin{equation}
\sum_{|\gamma|<T}e^{2\pi im\gamma}=\text{o}(N_{f}(T)),\label{Weyl appealing to our case!}
\end{equation}
where $N_{f}(T)$ denotes the number of zeros of the form $\beta+i\gamma$
in the rectangle with vertices $B\pm iT$ and $k+\frac{1}{2}-b\pm iT$.
Indeed,
\begin{align*}
\sum_{|\gamma|<T}e^{2\pi im\gamma} & =e^{-2\pi m\left(\frac{k}{2}+\frac{1}{4}\right)}\sum_{|\gamma|<T}e^{2\pi m\left(\frac{k}{2}+\frac{1}{4}+i\gamma\right)}\\
=e^{-2\pi m\left(\frac{k}{2}+\frac{1}{4}\right)}\sum_{|\gamma|<T} & \left(e^{2\pi m\left(\frac{k}{2}+\frac{1}{4}+i\gamma\right)}-e^{2\pi m\left(\beta+i\gamma\right)}\right)+e^{-2\pi m\left(\frac{k}{2}+\frac{1}{4}\right)}\sum_{|\gamma|<T}e^{2\pi m\left(\beta+i\gamma\right)}\\
=e^{-2\pi m\left(\frac{k}{2}+\frac{1}{4}-\beta\right)} & \sum_{|\gamma|<T}e^{2\pi im\gamma}\left(e^{2\pi m\left(\frac{k}{2}+\frac{1}{4}-\beta\right)}-1\right)+e^{-2\pi m\left(\frac{k}{2}+\frac{1}{4}\right)}\sum_{|\gamma|<T}e^{2\pi m\left(\beta+i\gamma\right)}.
\end{align*}

The second sum can be easily evaluated by appealing to Proposition \ref{landau proposition}. Indeed, since $m\in\mathbb{Z}\setminus\{0\}$, then $e^{2\pi m}\in\mathbb{R}^{+}\setminus\{0\}$.
Therefore, using (\ref{formula Landau explicit}) with $x=e^{2\pi m}$, we deduce that
\begin{equation*}
\sum_{|\gamma|<T}e^{2\pi m\left(\beta+i\gamma\right)} =\frac{b_{f}(d)}{\pi}\left(\delta_{e^{2\pi m},\lambda_{d}}-e^{2\pi m(k+\frac{1}{2})}\delta_{x,\lambda_{d}^{-1}}\right)T+O\left(\log(T)\right)=\text{o}\left(N_{f}(T)\right),
\end{equation*}
because $N_{f}(T)=\frac{2}{\pi}T\log(T)+O(T)$ [\cite{berndt_zeros_(i)}, Theorem 10]. Hence, to show (\ref{Weyl appealing to our case!}),
we just need to check that
\[
e^{-2\pi m\left(\frac{k}{2}+\frac{1}{4}-\beta\right)}\sum_{|\gamma|<T}e^{2\pi im\gamma}\left(e^{2\pi m\left(\frac{k}{2}+\frac{1}{4}-\beta\right)}-1\right)=\text{o}\left(N_{f}(T)\right).
\]

A sufficient conditions for this to happen can be obtained stated via the mean value theorem: indeed, since
\begin{align*}
\left|e^{-2\pi m\left(\frac{k}{2}+\frac{1}{4}-\beta\right)}\sum_{|\gamma|<T}e^{2\pi im\gamma}\left(e^{2\pi m\left(\frac{k}{2}+\frac{1}{4}-\beta\right)}-1\right)\right| & \leq e^{-2\pi m\left(\frac{k}{2}+\frac{1}{4}-\beta\right)}\sum_{|\gamma|<T}2\pi m\left|\beta-\frac{k}{2}-\frac{1}{4}\right|\max\left\{ 1,e^{2\pi m\left(\frac{k}{2}+\frac{1}{4}-\beta\right)}\right\} \\
 & \leq2\pi m\max\left\{ 1,e^{-2\pi m\left(\frac{k}{2}+\frac{1}{4}-\beta\right)}\right\} \,\sum_{|\gamma|<T}\left|\beta-\frac{k}{2}-\frac{1}{4}\right|,
\end{align*}
we see that, once we show the estimate
\[
\sum_{|\gamma|<T}\left|\beta-\frac{k}{2}-\frac{1}{4}\right|=\text{o}\left(N_{f}(T)\right),
\]
we are done. With the proof of this claim, we finish the proof of
our theorem. Thus, we shall prove:

\begin{claim}
The following estimate holds
\begin{equation}
\sum_{|\gamma|<T}\left|\beta-\frac{k}{2}-\frac{1}{4}\right|=\text{o}\left(N_{f}(T)\right). \label{avrg density hyp}
\end{equation}
\end{claim}

\begin{proof}
In order to prove this claim, we shall invoke Littlewood's lemma
(cf. [\cite{titchmarsh_zetafunction}, pp. 220-221]). Since $L(s,f)$ defines an analytic function
on the rectangle $\mathcal{R}=\left\{ s\in\mathbb{C}:\,\frac{k}{2}+\frac{1}{4}\leq\text{Re}(s)\leq a,\,\,-T<\text{Im}(s)\leq T\right\} $,
where $a>b>B$ is a positive real number which allows (\ref{cndition no zeros})
to happen, Littlewood's result can be applied to $L(s,f)$ in the domain $\mathcal{R}$. In fact, it states that 
\begin{equation}
\sum_{\beta>\frac{k}{2}+\frac{1}{4},\,|\gamma|<T}\left(\beta-\frac{k}{2}-\frac{1}{4}\right)=-\frac{1}{2\pi i}\,\intop_{\partial\mathcal{R}}\log\left(L(s,f)\right)\,ds.\label{Littlewood starting}
\end{equation}
Without any loss of generality, we may suppose that our choice of
$T$ satisfies $\min_{\gamma}\left\{ |\gamma+T|,|\gamma-T|\right\} \gg1/\log(T)$.
This choice comes at an expense of an error term of $O(\log(T))$ in the evaluation of the sum on the left-hand side of (\ref{Littlewood starting}) (cf. [\cite{berndt_zeros_(i)}, page 245, Theorem 3]). 
Continuing the development of the expression (\ref{Littlewood starting})
and taking into account that the sum on the left-hand side of (\ref{Littlewood starting})
is a real number, we are able to get the expression
\begin{align*}
\sum_{\beta>\frac{k}{2}+\frac{1}{4},\,|\gamma|<T}\left(\beta-\frac{k}{2}-\frac{1}{4}\right) & =\frac{1}{2\pi}\,\intop_{-T}^{T}\log\left|L\left(\frac{k}{2}+\frac{1}{4}+it,f\right)\right|\,dt-\frac{1}{2\pi}\intop_{-T}^{T}\log\left|L\left(a+it,f\right)\right|\,dt\\
 & -\frac{1}{2\pi}\intop_{\frac{k}{2}+\frac{1}{4}}^{a}\arg L(\sigma-iT,f)\,d\sigma+\intop_{\frac{k}{2}+\frac{1}{4}}^{a}\arg L(\sigma+iT,f)\,d\sigma.
\end{align*}

The estimate of the second integral is standard, because $a>B$ by
hypothesis and, according to (\ref{cndition no zeros}),
\begin{equation*}
\frac{1}{2\pi}\intop_{-T}^{T}\log\left|L\left(a+it,f\right)\right|\,dt=\frac{1}{2\pi}\intop_{-T}^{T}\log\left(\frac{|a_{f}(c)|}{c^{a}}\right)+\log\left|1+\frac{c^{a+it}}{a_{f}(c)}\sum_{n=c+1}^{\infty}\frac{a_{f}(n)}{n^{a+it}}\right|\,dt\ll T
\end{equation*}

The evaluation of the third and fourth integrals is a standard application
of Jensen's formula and we follow closely a similar argument outlined
in [\cite{berndt_zeros_(i)}, p. 252]: suppose that $\text{Re}\,L(\sigma+iT,f)$
has $N$ zeros for $\sigma\in\left[\frac{k}{2}+\frac{1}{4},a\right]$.
We then write $\left[\frac{k}{2}+\frac{1}{4},a\right]=\cup_{1\leq j\leq N+1}[\sigma_{j-1},\sigma_{j})$
in such a way that $\text{Re}\,L(\sigma+iT,f)$ has constant sign for
$\sigma\in[\sigma_{j-1},\sigma_{j})$. Since $\arg L(\sigma_{j}+iT,f)-\arg L(\sigma_{j-1}+iT,f)\leq\pi$,
we have that
\begin{equation}
|\arg L(\sigma+it,f)|\leq\pi(N+1)+\left|\arg L\left(\frac{k}{2}+\frac{1}{4}+it,f\right)\right|.\label{bound fot arg}
\end{equation}
In order to estimate $N$, consider the analytic function
\[
F(z):=\frac{1}{2}\left\{ L(z+iT,f)+\overline{L(\overline{z}+iT,f)}\right\} .
\]
Then one has that $F(\sigma)=\text{Re}\,L(\sigma+iT,f)$. Let $n_{F}(r)$
denote the number of zeros of $F(z)$ in the ball $|z-B|\leq r$.
Clearly, for any $R>0$,
\[
\intop_{0}^{2R}\frac{n_{F}(r)}{r}dr\geq\intop_{R}^{2R}\frac{n_{F}(r)}{r}dr\geq n_{F}(R)\,\log(2).
\]
Moreover, from Jensen's formula,
\[
n_{F}(R)\,\log(2)\leq\intop_{0}^{2R}\frac{n_{F}(r)}{r}dr=\frac{1}{2\pi}\intop_{0}^{2\pi}\log|F(B+2Re^{i\theta})|d\theta-\log|F(B)|,
\]
which implies that 
\begin{align}
n_{F}(R) & \leq\frac{1}{2\pi\log(2)}\intop_{0}^{2\pi}\log\left|F\left(B+2Re^{i\theta}\right)\right|d\theta-\frac{\log|F(B)|}{\log(2)}\nonumber \\
 & =\frac{1}{2\pi\log(2)}\intop_{0}^{2\pi}\log\left|F\left(B+2Re^{i\theta}\right)\right|d\theta+O(1),\label{after jensen}
\end{align}
where in the last step we have used the choice of $B$ in (\ref{cndition no zeros})
and the fact that the Dirichlet series defining $L(s,f)$ converges
absolutely for $\text{Re}(s)=B$. Now, let us take $R=B-a$ in (\ref{after jensen}): by the
Phragm\'en-Lindel\"of principle, we know that
\begin{equation}
\left|F(B+2Re^{i\theta})\right|\leq|T+2R|^{\frac{1}{2}+\epsilon}\ll T^{\frac{1}{2}+\epsilon},\label{lindelof inside integral}
\end{equation}
for any $\epsilon>0$. Combining (\ref{after jensen}) and (\ref{lindelof inside integral}),
we conclude that $n_{F}(R)=O(\log T)$. Since the open interval $(a,B)$
is clearly contained in the disc $|z-B|\leq R$, we have $N\leq n_{F}(R)$.
Returning to (\ref{bound fot arg}),
\[
|\arg L(\sigma+it,f)|\leq\pi(N+1)+O(1)\leq2\pi n_{F}(R)+O(1)=O(\log(T)),
\]
which proves the intermediate estimates 
\[
\intop_{\frac{k}{2}+\frac{1}{4}}^{a}\arg L(\sigma-iT,f)\,d\sigma\ll\log(T),\,\,\intop_{\frac{k}{2}+\frac{1}{4}}^{a}\arg L(\sigma+iT,f)\,d\sigma\ll\log(T).
\]

Therefore,
\begin{equation}
\sum_{\beta>\frac{k}{2}+\frac{1}{4},\,|\gamma|<T}\left(\beta-\frac{k}{2}-\frac{1}{4}\right)=\frac{1}{2\pi}\,\intop_{-T}^{T}\log\left|L\left(\frac{k}{2}+\frac{1}{4}+it,f\right)\right|\,dt+O(T)+O\left(\log(T)\right),\label{returning}
\end{equation}
and so it only remains to estimate the first integral above. This
can be done once we use Jensen's inequality and the mean value estimate
for $L(s,f)$, (\ref{mean value cusp form half}). A combination of both results yields
\begin{align*}
\frac{1}{2\pi}\,\intop_{-T}^{T}\log\left|L\left(\frac{k}{2}+\frac{1}{4}+it,f\right)\right|\,dt & \leq\frac{T}{\pi}\log\left(\frac{1}{2T}\intop_{-T}^{T}\left|L\left(\frac{k}{2}+\frac{1}{4}+it,f\right)\right|\,dt\right)\\
 & \leq\frac{T}{2\pi}\log\left(\frac{1}{2T}\,\intop_{-T}^{T}\left|L\left(\frac{k}{2}+\frac{1}{4}+it,f\right)\right|^{2}\,dt\right)\\
 & =\frac{T}{2\pi}\log\left(\frac{1}{T}\,\intop_{0}^{T}\left|L\left(\frac{k}{2}+\frac{1}{4}+it,f\right)\right|^{2}\,dt\right)\\
 & =\frac{T}{2\pi}\log\left(A_{f}\log(T)+O(1)\right)\\
 & =\frac{T}{2\pi}\log\log(T)+O(T).
\end{align*}

Returning to (\ref{returning}) and collecting all the available estimates,
we finally see that
\[
\sum_{\beta>\frac{k}{2}+\frac{1}{4},\,|\gamma|<T}\left(\beta-\frac{k}{2}-\frac{1}{4}\right)=\frac{T}{2\pi}\log\log(T)+O(T),
\]
from which we conclude that
\begin{equation*}
\sum_{|\gamma|<T}\left|\beta-\frac{k}{2}-\frac{1}{4}\right| =2\sum_{\beta>\frac{k}{2}+\frac{1}{4},\,|\gamma|<T}\left(\beta-\frac{k}{2}-\frac{1}{4}\right)=\frac{T}{\pi}\log\log(T)+O(T)=\text{o}\left(N_{f}(T)\right),
\end{equation*}
establishing the desired result.
\end{proof}

\textit{Acknowledgements:} This work was partially supported by CMUP, member of LASI, which is financed by national funds through FCT - Fundação para a Ciência e a Tecnologia, I.P., under the projects with reference UIDB/00144/2020 and UIDP/00144/2020. We also acknowledge the support from FCT (Portugal) through the PhD scholarship 2020.07359.BD. The author would like to thank to Semyon Yakubovich for unwavering support and guidance throughout the writing of this paper.

\begin{comment}
From this point on, we just need to find a suitable bound for $a_{2j}$.
Considering the nested sum above, we have two possibilities: if $1\leq k\leq j$
and $r_{k}\geq N_{0}$, we know by the contradiction hypothesis that
$\tau_{r_{k}}^{-2}\leq\frac{r_{k}^{-4}}{h^{2}}$, while if $1\leq r_{k}\leq N_{0}-1$,
we have that $\tau_{r_{k}}^{-2}<\frac{r_{k}^{-4}}{h^{\star2}}$ where
$h^{\star}:=\min_{1\leq n\leq N_{0}-1}\left\{ \frac{\tau_{n}}{n^{2}}\right\} $.
Since $N_{0}$ is the minimal integer for which $\tau_{n}\geq hn^{2}$,
we see that $h^{\star}<h$. Since we have at most $N_{0}-1$ terms
in (\ref{coefficient for the bound}) such that $r_{k}<N_{0}$, the
sequence $a_{2j}$ can be bounded in a simple form 
\begin{align*}
a_{2j} & \le\sum_{1\leq r_{1}...<r_{k}\leq N_{0}<r_{k+1}<...<r_{j}}\frac{1}{\tau_{r_{1}}^{2}\cdot...\cdot\tau_{r_{j}}^{2}}\leq\frac{1}{h^{\star2k}h^{2j-2k}}\sum_{1\leq r_{1}<r_{2}<...<r_{j}}\frac{1}{r_{1}^{4}\cdot...\cdot r_{j}^{4}}\\
 & =\left(\frac{h}{h^{\star}}\right)^{2k}\frac{1}{h^{2j}}\sum_{1\leq r_{1}<r_{2}<...<r_{j}}\frac{1}{r_{1}^{4}\cdot...\cdot r_{j}^{4}}\leq\left(\frac{h}{h^{\star}}\right)^{2N_{0}}\frac{1}{h^{2j}}\sum_{1\leq r_{1}<r_{2}<...<r_{j}}\frac{1}{r_{1}^{4}\cdot...\cdot r_{j}^{4}}.
\end{align*}

\end{comment}

\end{document}